\PassOptionsToPackage{dvipsnames,svgnames}{xcolor}
\documentclass[a4paper]{article}
\pdfoutput=1
\usepackage{fullpage}

\usepackage[utf8]{inputenc}
\usepackage[english]{babel}

\usepackage{bbm}
\usepackage{adjustbox}
\usepackage{amsmath}
\usepackage{amssymb}
\usepackage{amsthm}
\usepackage{mathtools}
\usepackage{stmaryrd}
\usepackage{longtable}
\usepackage{float}
\usepackage{hhline}
\usepackage{dsfont}
\usepackage{accents}
\usepackage{hyperref}
\usepackage{cleveref}
\crefname{subsection}{subsection}{subsections}
\crefname{subsubsection}{subsubsection}{subsubsections}
\crefname{subsubsubsection}{subsubsubsection}{subsubsubsections}
\usepackage{tableof}
\usepackage{natbib}
\usepackage{graphicx}
\usepackage{multirow}
\usepackage{bm}
\usepackage{import}
\usepackage{mathdots}
\usepackage{forest}
\usepackage{pgfplots}   
\pgfplotsset{compat=1.16}
\usepackage{schemabloc}
\usepackage{braket}
\usepackage{xparse}
\usepackage{eurosym}

\usepackage[colors]{optsys}
\usepackage{xcolor}
\usepackage{tikz}
\usepackage{hyperref}

\hypersetup{
    colorlinks=true,
    linkcolor=blue,
    filecolor=magenta,      
    urlcolor=cyan,
}

\usepackage{footnote}

\usepackage{subfiles}
\usepackage{appendix}

\usepackage{xspace}

\usepackage[ruled,vlined]{algorithm2e}

\usetikzlibrary{decorations.pathreplacing}
\tikzset{%
  node distance=5cm,
  every edge quotes/.append style={midway, below},
}
\tikzset{
decision/.style={rectangle, minimum height=3pt, minimum width=30pt, draw=black, thick, inner sep=3pt},
chance/.style={circle, minimum width=30pt, draw=black, thick, inner sep=0pt}
 }

\forestset{
 declare toks={optimality}{},
 declare toks={edge label below}{},
 declare toks={edge label alea}{},
 declare toks={edge label proba}{},
 }

\usepackage[autolanguage]{numprint}
\renewcommand{\graph}{\mathcal{G}}

\newcommand{\vertexindex}{v}
\newcommand{\vertexset}{\mathbb{V}}

\newcommand{\arcindex}{e}
\newcommand{\arcset}{\mathbb{E}}

\newcommand{\timeindex}{t}
\newcommand{\cumulatedtimeindex}{s}
\newcommand{\timeset}{\mathbb{T}}
\newcommand{\oilsuper}{^{\textsc{o}}}
\newcommand{\gassuper}{^{\textsc{g}}}
\newcommand{\watersuper}{^{\textsc{w}}}
\newcommand{\disgassuper}{^{\textsc{dg}}}
\newcommand{\poresuper}{^{\textsc{p}}}
\newcommand{\oilunder}{_{\textsc{o}}}
\newcommand{\gasunder}{_{\textsc{g}}}
\newcommand{\waterunder}{_{\textsc{w}}}

\newcommand{\totalflowvar}{F}
\newcommand{\oilproduced}{\totalflowvar\oilsuper}
\newcommand{\gasproduced}{\totalflowvar\gassuper}
\newcommand{\waterproduced}{\totalflowvar\watersuper}

\newcommand{\waterinjected}{\totalflowvar^{wi}}

\newcommand{\InflowPerformanceRelationship}{\textsc{Ipr}}

\newcommand{\reservoirvolume}{V}
\newcommand{\totalreservoiroil}{\reservoirvolume\oilsuper}
\newcommand{\totalreservoirgas}{\reservoirvolume\gassuper}
\newcommand{\totalreservoirwater}{\reservoirvolume\watersuper}
\newcommand{\dissolvedgas}{\reservoirvolume\disgassuper}
\newcommand{\porevolume}{\reservoirvolume\poresuper}

\newcommand{\saturation}{S}
\newcommand{\oilsaturation}{\saturation\oilsuper}
\newcommand{\gassaturation}{\saturation\gassuper}
\newcommand{\watersaturation}{\saturation\watersuper}

\newcommand{\presvar}{P}
\newcommand{\respressure}{\presvar^{\textsc{r}}} % reservoir pressure
\newcommand{\oilformationfactor}{B\oilunder}

\newcommand{\gasformationfactor}{B\gasunder}

\newcommand{\solutiongasoilratio}{R_s}

\newcommand{\waterformationfactor}{B\waterunder}

\newcommand{\porecompressibility}{c_f}

\newcommand{\watercut}{w^{\textsc{wc}}}
\newcommand{\watercutfunct}{\textsc{Wct}}

\newcommand{\states}{x}

\newcommand{\controls}{u}

\newcommand{\admcontrolset}{\mathcal{U}^{ad}}
\newcommand{\controlseries}{\tilde{\mathbf{U}}}

\renewcommand{\dynamics}{f}

\newcommand{\pipeboolvar}{o}

\newcommand{\lambertfunct}{\mathcal{W}}

\newcommand{\costfunct}{\mathcal{L}}
\newcommand{\finalcost}{\mathcal{K}}

\newcommand{\bellval}{\mathcal{J}}
\newcommand{\policyinst}{\mu}
\renewcommand{\policy}{\pi}

\newcommand{\declinecurve}{h}

\newcommand{\GeneralProductionFunction}{\Phi}

\usepackage{dcolumn}
\newcolumntype{d}[1]{D{.}{.}{#1}}

\newcommand{\toprule}{\hline\hline}
\newcommand{\midrule}{\hline}
\newcommand{\bottomrule}{\hline\hline}

\date{\today}

\def\imageonetanks{1_tank_}
\def\imagestwotanks{2_tanks_}
\def\imageswaterinj{deterministic_water_inj_}

\renewcommand{\bleue}[1]{#1}

\title{Multistage Optimization of a Petroleum Production System with Material Balance
  Model}
\author{
Cyrille Vessaire\footnote{CERMICS, Ecole des Ponts, Marne-la-Vall\'ee, France},
\and Jean-Philippe Chancelier\footnotemark[1],
\and Michel {De Lara}\footnotemark[1],
\and Pierre Carpentier\footnote{UMA, ENSTA Paris, Institut Polytechnique de Paris,
  Palaiseau, France},
\and Alejandro {Rodr\'iguez-Mart\'inez}\footnote{IAM, TotalEnergies SE,
  Pau, France},
\and Anna Robert\footnote{OneTech/R\&D, TotalEnergies SE, Saclay, France}
}

\begin{document}

\maketitle

\begin{abstract}
  % \bleue{
    In this paper, we propose a mathematical formulation for the optimal management
    over time of an oil production network as a multistage optimization problem. The
    proposed model differs from common practice where the reservoir of the oil production
    network is approximated by decline curves or by black-box simulators. We model the
    reservoir as a controlled (non-linear) dynamical system by using material balance
    equations, under the standard assumptions that the fluids follow a black-oil model and
    that the reservoir has a tank-like behavior. The state of the dynamical system has
    five dimensions: the total volume of respectively oil, gas, and water in the
    reservoir; the total pore volume; and the reservoir pressure. We use a dynamic
    programming algorithm to numerically solve the multistage optimization problem on two
    specific instances of the general optimization problem where the state dimension can
    be reduced from dimension five to dimension one or two. More precisely, the first
    numerical application consists in optimizing the production of a dry gas reservoir
    which is subdivided in two tanks and which leads to a two-dimensional state (one
    dimension per tank), whereas the second numerical application tackles an oil reservoir
    with water injection which leads to a two-dimensional state. The two applications
    illustrate that our approach handles interconnected tanks (in the dry gas case) and
    that our approach allows optimization beyond first recovery of oil (in the oil with
    water injection case). We also provide numerical and theoretical comparisons with
    decline curves in the dry gas application.
  % }
\end{abstract}

\section{Introduction}
\label{sect:Intro}

Oil and gas projects usually span over several decades and involve complex planning and
decision-making. Therefore, multistage optimization is a relevant tool to address the
long-term performance of such projects. This is the focus of this paper.

The lifetime of a field usually consists of five phases: exploration, where reservoirs
containing hydrocarbon are found; appraisal, to give a value to a field; development,
where infrastructures are planned and installed; production, where hydrocarbon is finally
produced; abandonment, where the field stops producing and infrastructures are
decommissioned and removed.
% , restoring the environment to its original state.
In this paper, we focus on the production phase. We consider that the infrastructure has
already been installed in the development phase, and we thus focus on finding a production
schedule that maximizes the profit over the full production phase.

Now, we position our contribution with respect to the currently available literature.
According to the survey~\citep{khor_oil_2017_review}, there is extensive research on how to
optimize the production phase, with multiple approaches. The authors present three main
methods for the optimization of petroleum production systems: sensitivity analysis by
employing simulation tools, heuristic rules and mathematical optimization, the approach of
this paper. Most of the literature resorts to the first two approaches.

Regarding mathematical optimization, most works on the topic have considered black-box
simulators to describe the reservoir
% behavior:
dynamics: \citet{hepguler_integration_softwares} consider integrating both a network model
and a proprietary reservoir model (a commercial simulation software for reservoir
modeling); \cite{gerogiorgis_dynamic_2006} combine a proprietary reservoir simulator with
a general optimization formulation. In~\cite{sarma_efficient_2006} a closed-loop
multistage optimal control approach with a simulator that can be updated with new data
from sensors is considered. It is also a standard practice to add some optimization layer
over a commercial reservoir simulator to locally improve a production planning, such as
modifying the pressure on different points of the petroleum production system to locally
improve an operational solution (see ECLIPSE by Schlumberger, or GAP and MBAL by Petroleum
Experts).
% \bleue{
  In theory, such approach could be amenable to dynamic programming.
  However, this is not done in practice due to the the computation time of a single
  simulation run.
% }

A limited fraction of the literature addresses the problem as a multistage optimization
problem, such as in~\citet{Grossmann98,GUPTA2012,Marmier}. In those papers, the
formulation relies on dynamical models based on decline curves (or type curves). In short,
decline curves are functions that take as input the cumulative production and return the
maximal well rate. In the context of mathematical optimization, decline curves were first
assumed to be linear, such as in~\cite{bohannon_linear_1970}, before being assumed to be
piecewise linear in~\cite{frair_economic_1975} or polynomial in~\cite{Marmier}, % \bleue{,
  or
  even being assumed to be given by a set of logical relationships for shale gas
  in~\cite{hong_optimal_2020},
  %}
  when algorithms could treat those refinements. The decline
curves are generally constructed by using a foresight of the optimal solution that is
looked after, as they are usually generated by assuming a production schedule.
In~\cite{SATTER2016}, the authors write that, usually, decline curves analysis is
performed under one key assumption: the wells produce at ``constant bottom-hole
pressure''. They also state that ``in reality, such a condition may not be observed''.
Note that decline curves can, in some cases, provide an accurate representation of the
reservoir if the wells that constitute the oil field are independent of each other, and
when we are only considering first recovery of oil and gas (i.e. when we are only
producing fluids in the reservoir and without any injection of gas and water in the
reservoir).
% \bleue{
  Despite those shortcomings, mathematical formulations using decline
  curves are commonly used in oilfield development studies. For example, two
  case studies, one in Brazil~\citep{silva_oilfield_2021} and one in New
  Mexico~\citep{davis_optimal_2021} use decline curves to solve a multistage
  optimization problem.%}

% \bleue{
  Part of the literature also tries to develop a middle ground between using a
  black-box reservoir simulator and using decline curves. For example, some papers use
  parametrized surrogate models (also called proxy models). Parameters of the surrogate
  model are to be adjusted to fit simulators output or real data
  (see~\citep{caballero_algorithm_2008}). Numerous applications following the methodology
  developed in~\cite{caballero_algorithm_2008} have been done, each one being
  characterized by a specific surrogate model: in~\cite{lei_formulations_2022}, 
  a proxy model (presented in~\cite{lei_compact_2021}) that takes into account
  the decommission timing and costs in the development planning is used; whereas
  in~\cite{camponogara_integrated_2017}, the authors use MILP as a proxy model and apply
  it to a case in the Santos Basin; finally, in~\cite{moolya_optimal_2022}, the authors
  also use a MILP surrogate model combined with aggregation and disaggregation methods
  in well placement problems. In~\cite{epelle_computational_2020}, the authors compare the
  performances between MILP and MINLP formulations of the surrogate model.%}

In this paper, we represent the reservoir as a controlled dynamical system based on
black-oil model and conservation laws (mass balance equations) for a tank-like reservoir
instead of using decline curves or surrogate models based on a reservoir simulator. Mass
balance equations belong nowadays to the folklore of petroleum engineering and have been
described many times in the reservoir modeling literature (see~\cite{Dake}). We formulate
the management problem as a multistage optimization problem, and we use the dynamic
programming algorithm to solve it (see~\cite{bertsekas_dynamic}). To the best of our
knowledge, this approach is new in the oil and gas literature. This formulation is well
adapted to first and secondary recovery of oil and gas cases. Moreover, multistage
optimization and dynamic programming are well adapted to tackle more complex formulations
with uncertain parameters and partial observations.
%We will exploit this potential in future works.

\section{Formulation of the management of a petroleum production system as a multistage
  optimization problem}
We consider a production system composed of a reservoir and production assets (pipes,
wells, chokes). We represent the topology of the production assets as a simple graph
$\graph = (\vertexset, \arcset)$, where $\vertexset$ is the set of vertices and
$\arcset \subset \vertexset^2$ is the set of edges. Controls are variables indexed by
either vertices or edges. We place the different production assets on the graph, with the
pipes as the edges of the graph, and the rest of the assets, such as the well-heads,
positioned on the vertices of the graph. This is illustrated in
Figure~\ref{fig:drawing_network}. The wells' perforations are represented as vertices
($w_i$ in Figure~\ref{fig:drawing_network}) where the fluids produced enter the graph. On
the other vertices, we have assets such as the well-head chokes ($wh_i$ in
Figure~\ref{fig:drawing_network}), or joints between different pipes (noted $i_1$). We can
also have valves to open or close pipes. Finally, we have an export point (on the vertex
$e$).

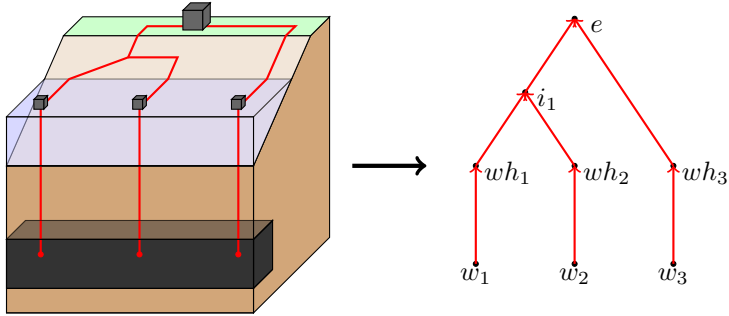
\begin{figure*}[!htp]
  \centering
  \begin{tikzpicture}[scale=0.65]
    %%%%%%%%%%%%%%%%%%%%%%%%%%%%%%%%%%%
    % drawing of a production network %
    %%%%%%%%%%%%%%%%%%%%%%%%%%%%%%%%%%%

    % classic rocks coordinate
    \coordinate (bottright) at (2.5, -4, 0);
    \coordinate (bottrightsea) at (2.5, -1, 0);
    \coordinate (bottrightsshore) at (2.5, 0.5, -3);
    \coordinate (endrightsshore) at (2.5, 0.5, -4);
    \coordinate (endright) at (2.5, -4, -4);

    \coordinate (bottleft) at (-2.5, -4, 0);
    \coordinate (bottleftsea) at (-2.5, -1, 0);
    \coordinate (bottleftsshore) at (-2.5, 0.5, -3);
    \coordinate (endleftsshore) at (-2.5, 0.5, -4);
    \coordinate (endleft) at (-2.5, -4, -4);

    % draw classic rocks
    % \filldraw[brown!50,draw=black] (bottleft) -- (bottleftsea)
    % -- (bottleftsshore) -- (endleftsshore) -- (endleft) -- cycle;
    \filldraw[brown!70,draw=black] (bottright) -- (bottrightsea)
    -- (bottrightsshore) -- (endrightsshore) -- (endright) -- cycle;
    \filldraw[brown!70,draw=black] (bottleft) -- (bottleftsea)
    -- (bottrightsea) -- (bottright) -- cycle;
    \filldraw[brown!20,draw=black] (bottleftsea) -- (bottleftsshore)
    -- (bottrightsshore) -- (bottrightsea) -- cycle;
    \filldraw[green!20,draw=black] (bottleftsshore) -- (endleftsshore)
    -- (endrightsshore) -- (bottrightsshore) -- cycle;

    % sea coordinates
    \coordinate (sealvlright) at (2.5, 0, 0);
    \coordinate (endsealvlright) at (2.5, 0, -2);
    \coordinate (endsealvlleft) at (-2.5, 0, -2);
    \coordinate (sealvlleft) at (-2.5, 0, 0);

    % draw sea
    \filldraw[blue!20,draw=black, opacity=0.5] (bottrightsea) -- (sealvlright)
    -- (endsealvlright) -- cycle;
    \filldraw[blue!20,draw=black, opacity=0.5] (bottleftsea) -- (sealvlleft)
    -- (endsealvlleft) -- cycle;
    \filldraw[blue!20,draw=black, opacity=0.5] (bottleftsea) -- (sealvlleft)
    -- (sealvlright) -- (bottrightsea) -- cycle;
    \filldraw[blue!20,draw=black, opacity=0.5] (sealvlleft) -- (endsealvlleft)
    -- (endsealvlright) -- (sealvlright) -- cycle;

    % reservoir coordinate
    \coordinate (bottrightres) at (2.5, -3.5, 0);
    \coordinate (bottendrightres) at (2.5, -3.5, -1);
    \coordinate (toprightres) at (2.5, -2.5, 0);
    \coordinate (topendrightres) at (2.5, -2.5, -1);
    \coordinate (bottleftres) at (-2.5, -3.5, 0);
    \coordinate (bottendleftres) at (-2.5, -3.5, -1);
    \coordinate (topleftres) at (-2.5, -2.5, 0);
    \coordinate (topendleftres) at (-2.5, -2.5, -1);

    % draw reservoir
    % \filldraw[black!20,draw=black, opacity=0.5] (bottleftres) -- (bottendleftres)
    % -- (topendleftres) -- (topleftres) -- cycle;
    \filldraw[black!80,draw=black] (bottrightres) -- (bottendrightres)
    -- (topendrightres) -- (toprightres) -- cycle;
    \filldraw[black!80,draw=black] (bottleftres) -- (topleftres)
    -- (toprightres) -- (bottrightres) -- cycle;
    \filldraw[black!80,draw=black, opacity=0.5] (topleftres) -- (topendleftres)
    -- (topendrightres) -- (toprightres) -- cycle;

    % Bottom well
    \coordinate (w1) at (-2, -3, -0.5);
    \coordinate (w2) at (0, -3, -0.5);
    \coordinate (w3) at (2, -3, -0.5);

    % top well
    \coordinate (wh1) at (-2, 0, -0.5);
    \coordinate (wh2) at (0, 0, -0.5);
    \coordinate (wh3) at (2, 0, -0.5);

    \filldraw[red] (w1) circle (0.05cm);
    \filldraw[red] (w2) circle (0.05cm);
    \filldraw[red] (w3) circle (0.05cm);

    \draw[red, thick] (w1) -- (wh1);
    \draw[red, thick] (w2) -- (wh2);
    \draw[red, thick] (w3) -- (wh3);

    % exit point
    \coordinate (exit) at (0, 0.5, -3.5);

    % draw pipe
    \draw[red, thick] (wh1) -- ++(0,0,-1.5) -- ++(1,0.25,-0.5)
    -- ++(0,0.25,-0.5) -- ++(0,0,-0.5) -- (exit);
    \draw[red, thick] (wh2) -- ++(0,0,-1.5) -- ++(0,0.25,-0.5)
    -- ++(-1,0,0);
    \draw[red, thick] (wh3) -- ++(0,0,-1.5) -- ++(0,0.5,-1)
    -- ++(0,0,-0.5) -- (exit);

    % draw platforms
    \filldraw[black!60, draw=black] (wh1)++(0.1,0,0.1) -- ++(0,0,-0.2)
    -- ++(0,0.2,0) -- ++(0,0,0.2) -- cycle;
    \filldraw[black!60, draw=black] (wh1)++(0.1,0.2,0.1) -- ++(0,0,-0.2)
    -- ++(-0.2,0,0) -- ++(0,0,0.2) -- cycle;
    \filldraw[black!60, draw=black] (wh1)++(0.1,0.2,0.1) -- ++(-0.2,0,0)
    -- ++(0,-0.2,0) -- ++(0.2,0,0) -- cycle;

    \filldraw[black!60, draw=black] (wh2)++(0.1,0,0.1) -- ++(0,0,-0.2)
    -- ++(0,0.2,0) -- ++(0,0,0.2) -- cycle;
    \filldraw[black!60, draw=black] (wh2)++(0.1,0.2,0.1) -- ++(0,0,-0.2)
    -- ++(-0.2,0,0) -- ++(0,0,0.2) -- cycle;
    \filldraw[black!60, draw=black] (wh2)++(0.1,0.2,0.1) -- ++(-0.2,0,0)
    -- ++(0,-0.2,0) -- ++(0.2,0,0) -- cycle;

    \filldraw[black!60, draw=black] (wh3)++(0.1,0,0.1) -- ++(0,0,-0.2)
    -- ++(0,0.2,0) -- ++(0,0,0.2) -- cycle;
    \filldraw[black!60, draw=black] (wh3)++(0.1,0.2,0.1) -- ++(0,0,-0.2)
    -- ++(-0.2,0,0) -- ++(0,0,0.2) -- cycle;
    \filldraw[black!60, draw=black] (wh3)++(0.1,0.2,0.1) -- ++(-0.2,0,0)
    -- ++(0,-0.2,0) -- ++(0.2,0,0) -- cycle;

    % draw exit
    \filldraw[black!60, draw=black] (exit)++(0.2,0,0.2) -- ++(0,0,-0.4)
    -- ++(0,0.4,0) -- ++(0,0,0.4) -- cycle;
    \filldraw[black!60, draw=black] (exit)++(0.2,0.4,0.2) -- ++(0,0,-0.4)
    -- ++(-0.4,0,0) -- ++(0,0,0.4) -- cycle;
    \filldraw[black!60, draw=black] (exit)++(0.2,0.4,0.2) -- ++(-0.4,0,0)
    -- ++(0,-0.4,0) -- ++(0.4,0,0) -- cycle;

    \draw[black, ultra thick, ->] (4.5,-1)--(6,-1);
    
    \draw[black] (bottleft) -- (bottleftsea)
    -- (bottrightsea) -- (bottright) -- cycle;

    % draw sea
    \draw[black] (bottleftsea) -- (sealvlleft)
    -- (sealvlright) -- (bottrightsea) -- cycle;

    % draw reservoir
    % \filldraw[black!20,draw=black, opacity=0.5] (bottleftres) -- (bottendleftres)
    % -- (topendleftres) -- (topleftres) -- cycle;
    \draw[black] (bottrightres) -- (bottendrightres)
    -- (topendrightres) -- (toprightres) -- cycle;
    \draw[black] (bottleftres) -- (topleftres)
    -- (toprightres) -- (bottrightres) -- cycle;
    \draw[black, opacity=0.5] (topleftres) -- (topendleftres)
    -- (topendrightres) -- (toprightres) -- cycle;

    %%%%%%%%%%%%%%
    % Draw graph %
    %%%%%%%%%%%%%%

    \filldraw[black] (7,-3) circle (0.05cm);
    \node at (7, -3.25) {\noir{$w_1$}};
    \filldraw[black] (9,-3) circle (0.05cm);
    \node at (9, -3.25) {\noir{$w_2$}};
    \filldraw[black] (11,-3) circle (0.05cm);
    \node at (11, -3.25) {\noir{$w_3$}};
    % \filldraw[black] (3,-1) circle (0.05cm);
    % \node at (3, -1.15) {$w_4$};
    \filldraw[black] (7,-1) circle (0.05cm);
    \node at (7.65, -1.15) {$\bleue{wh_1}$};
    \filldraw[black] (9,-1) circle (0.05cm);
    \node at (9.65, -1.15) {$\bleue{wh_2}$};
    \filldraw[black] (11,-1) circle (0.05cm);
    \node at (11.65, -1.15) {$\bleue{wh_3}$};
    % \filldraw[black] (3,0) circle (0.05cm);
    % \node at (3.4, -0.15) {$\bleue{wh_4}$};
    \filldraw[black] (8,0.5) circle (0.05cm);
    \node at (8.45, 0.35) {$\bleue{i_1}$};
    \filldraw[black] (9,2) circle (0.05cm);
    \node at (9.45, 1.85) {\noir{$e$}};
    % \filldraw[black] (1.5,3) circle (0.05cm);
    % \node at (1.8, 2.85) {$e$};
    \draw[red][<-, thick]  (7,-1)--(7,-3);
    \draw[red][<-, thick]  (9,-1)--(9,-3);
    \draw[red][<-, thick]  (11,-1)--(11,-3);
    % \draw[red][<-, thick]  (3,0)--(3,-1);
    \draw[red][->, thick]  (7,-1)--(8,0.5);
    \draw[red][->, thick] (9,-1)--(8,0.5);
    \draw[red][->, thick] (8,0.5)--(9,2);
    % \draw[red][->, thick] (1,2)--(1.5,3);
    \draw[red][->, thick] (11,-1)--(9,2);
    % \draw[red][->, thick] (3,0)--(1.5,3);
    % \node at (-1.5,0) {$\bleue{\intersectionset} \subset \vertexset$};

  \end{tikzpicture}
  \caption{Representing a production network as a graph}
  \label{fig:drawing_network}
\end{figure*}

All the relevant operational constraints and features - such as pressure loss on the
pipes, mass balance of the fluids at each vertex, allowed pressures, and flow rate ranges
in the different assets or unavailability due to maintenance - are modeled as constraints
using variables defined on the edges and vertices of the graph. Indeed, the graph allows
us to define the different controls we can apply to the system, such as opening or closing
valves or changing the well-head pressures.
% We can add standard formulation for the network, with constraints such as pressure loss on
% the pipes, mass balance of the fluids at each nodes, etc...
Detailed formulations on the production network can be seen in~\citep{GUPTA2012}. We will not explicit it in the general case as
this is not our main focus, and we only present numerical applications without taking into
account the production network.

% \jpc{c'est un peu bizarre alors d'introduire un graphe
%   formellement si on s'en sert pas}.

As we aim to optimize the system over the whole production phase (i.e. over multiple
years), we consider multiple time steps belonging to a finite set
$\timeset = \na{0, 1, 2, \dots, \horizon}$ where the parameter $\horizon$ is a natural
number. Those time steps are usually monthly\footnote{Numerical applications will be done
  with monthly time steps and a horizon $\horizon$ of 15 or 20 years}, but under certain
conditions other time steps may be considered.

We propose (and are going to detail) a general formulation of the petroleum production
system optimization problem as follows
\begin{subequations}
  \begin{align}
      \bellval^{\star}(\states_0) = \max_{\states, \controls} ~ ~&
      \sum_{\timeindex=0}^{\horizon-1} \rho^{\timeindex} \costfunct_{\timeindex}
      (\states_{\timeindex}, \controls_{\timeindex})
      + \rho^{\horizon} \finalcost(\states_{\horizon})
      \label{eq:obj_general}\\
      %%%%
      s.t. ~~&
      \states_{0} ~ given \eqfinv
      \label{eq:state_initialisation} \\
      %%%%
      &
      \states_{\timeindex+1} =
      \dynamics(\states_{\timeindex}, \controls_{\timeindex})
      \eqsepv \forall \timeindex \in \timeset \setminus \{\horizon\}
      \eqfinv
      \label{eq:dynamics_general}\\
      %%%%
      &
      \controls_{\timeindex} \in
      \admcontrolset_{\timeindex}(\states_{\timeindex})
      \eqsepv \forall \timeindex \in \timeset \setminus \{\horizon\}
      \eqfinp
      \label{eq:controls_admissibility_general}
  \end{align}
  \label{eq:general_formulation}
\end{subequations}

The variables in Problem~\eqref{eq:general_formulation} are: \textit{i}) the state of the
reservoir~$\states_{\timeindex} \in \XX \subset \RR^n$ (with $\XX$ the state space);
\textit{ii}) the controls~$\controls_{\timeindex} \in \UU \subset \RR^p$ (with $\UU$ the
control space), which are the decisions that can be taken at time step~$\timeindex$ (for
example, the pressure $\presvar_{\vertexindex, \timeindex}$ at the different vertices
$\vertexindex \in \vertexset$ of the graph, and the Boolean
$\pipeboolvar_{\arcindex, \timeindex}$ stating if a pipe $\arcindex \in \arcset$ of the
graph is opened or closed). The reservoir is defined as a controlled dynamical system,
with state~$\states_{\timeindex}$, control~$\controls_{\timeindex}$ and an evolution
function $\dynamics$ of the controlled dynamical system, whose construction is the focus
of Section~\ref{sect:Reservoir_Dynamical_System}. At every time step~$\timeindex$, when
the decision maker takes decision~$\controls_{\timeindex}$, an instantaneous gain denoted
by $\costfunct_{\timeindex}(\states_{\timeindex}, \controls_{\timeindex})$ occurs. In the
last stage, the final state~$\states_{\horizon}$ is valued as
$\mathcal{K}(\states_{\horizon})$. We denote by $\rho$ the discount factor. We finally
obtain the objective function seen to the right of the max in
Equation~\eqref{eq:obj_general} by adding all terms. The known initial state of the
reservoir is defined in Equation~\eqref{eq:state_initialisation}. The controlled dynamics
of the reservoir is given in Equation~\eqref{eq:dynamics_general}.
% \cyrille{Changement demandé par Alejandro}
% Equation~\eqref{eq:controls_admissibility_general} states that the allowed controls belong
% to an admissibility set, which is for each time step~$\timeindex$ a set-valued mapping
% which takes a given state $\states_{\timeindex}$ of the reservoir and returns the set of
% allowed controls.
Equation~\eqref{eq:controls_admissibility_general} states that, at each time
step~$\timeindex$, the allowed controls belong to an admissibility set that depends
on~$\states_{\timeindex}$. The dependence is noted by
$\admcontrolset_{\timeindex}(\states_{\timeindex})$, which is for each time
step~$\timeindex$ a set-valued mapping that takes a given state~$\states_{\timeindex}$ of
the reservoir and returns the set of allowed controls. As far as the petroleum
application is concerned, the admissibility set notably depends on
the reservoir pressure, which constrains the different pressures in the petroleum
production system. It also depends on the production network itself:
some pipes can be controlled, while others cannot; facilities have planned or unplanned
downtimes, etc.
% Here, we focus on the admissibility set specific to the numerical
% applications.
Extensive formulations of the admissibility set of the production
depending on the reservoir pressure can be seen in~\cite{Grossmann98}.

The petroleum production system optimization problem, as formulated
in~\eqref{eq:general_formulation}, is a classical deterministic discrete time optimal
control problem.
%It is to be understood as an optimization problem under full
%information, the decision maker has access at each time $t$ to the history of
%past controls and past states of the system to take his decisions.
It is known that this problem can be solved by dynamic programming and that the resulting optimal
control at time $t$ is a function of the current state at time $t$.

In order to solve Problem~\eqref{eq:general_formulation}, we use a family of value
functions $\bellval_{\timeindex}:\XX \mapsto \RR$, where we recall that $\XX$ is the state
space. We call
\emph{policy}~$\policyinst = \left\{ \policyinst_{0}, \dots,
  \policyinst_{\horizon-1}\right\}$ a set of
mappings~$\policyinst_{\timeindex}:\XX \to \UU$ from states~$\states$ into
admissible controls $\controls$. We have the following proposition
(see~\citep[Chap. 1]{bertsekas_dynamic}).
\begin{proposition}
  For every initial state~$\states_0 \in \XX$, the optimal cost $\bellval^*(\states_0)$ of
  Problem~\eqref{eq:general_formulation} is equal to $\bellval_{0}(\states_0)$, given by
  the last step of the following algorithm, which proceeds backward in time from final
  time step $\horizon$ to initial time step $0$:
  \begin{subequations}
    \begin{align}
      \bellval_{\horizon}(\states)
      & = \rho^{\horizon} \finalcost(\states)
        \eqsepv \forall \states \in \XX
        \eqsepv
        \label{eq:prog_dyn_final}
      \\
      \bellval_{\timeindex}(\states)
      & =
        \max_{ \controls \in \admcontrolset_{\timeindex}(\states)}
        \Big(
        \rho^{\timeindex} \costfunct_{\timeindex}(\states,
        \controls)
        \nonumber \\
      & \hspace{0cm}
        + \bellval_{\timeindex+1}\left(
        \dynamics(\states, \controls)\right)
        \Big)
        \eqsepv \forall \states \in \XX,
        \forall \timeindex \in \timeset \setminus \left\{ \horizon \right\}
        \eqfinp
        \label{eq:prog_dyn_rec}
    \end{align}
  \end{subequations}
  Furthermore, if
  $\controls^{\star}=\policyinst_{\timeindex}^{\star}(\states)$
  maximizes the right-hand side of~\eqref{eq:prog_dyn_rec} for each $\states$
  and $\timeindex$, then the policy $\policyinst^{\star}= \left\{\policyinst_{0}^{\star},
  \dots, \policyinst_{\horizon-1}^{\star} \right\}$ is optimal.
\end{proposition}

To solve Problem~\eqref{eq:general_formulation}, we compute $\bellval_{0}$. To do so, we
use a dynamic programming algorithm (see Algorithm~\ref{alg:dynamic_programming}). For
that purpose, we discretize the controls, that now belong to a finite set denoted
by~$\UU_d$, and the states that belong to a finite set~$\XX_d$. Numerically, we also use a
multi-linear interpolation for the value functions between the states.
%  We also consider that the value functions follow a multi-linear interpolation between the states.
\begin{algorithm}[htp!]
  \For{$\states \in \XX_d$}
  {
    $\bellval_{\horizon}(\states) = \rho^{\horizon} \finalcost(\states)$\;
  }

  \For{$\timeindex =  \horizon-1, \ldots, 1$}
  {
    \For{$\states \in \XX_d$}
    {
      best\_value = - $\infty$\;
      best\_controls = 0\;
      \For{$\controls \in \UU_d$}
      {
        current\_value = $\rho^{\timeindex} \costfunct_{\timeindex}(\states,
        \controls) + \bellval_{\timeindex+1}\Bp{\dynamics \bp{\states, \controls} }$\;
        \If{\emph{current\_value} $\geq$ \emph{best\_value}}
        {
          best\_value = current\_value\;
          best\_controls = $\controls$\;
        }
      }
      $\bellval_{\timeindex}(\states) = $best\_value\;
      $\policyinst_{\timeindex}(\states) = $best\_controls\;
    }
  }

  \Return{$\bp{\bellval_{\timeindex}, \policyinst_{\timeindex}}_{{\timeindex}\in\timeset}$}
  \caption{dynamic programming algorithm used to solve
    Problem~\eqref{eq:general_formulation}}
  \label{alg:dynamic_programming}
\end{algorithm}

\section{Formulation of the reservoir extraction as a controlled dynamical system}
\label{sect:Reservoir_Dynamical_System}
% Section where we detail the reservoir model

In this section, we show how to represent the time evolution of the reservoir as a
dynamical system, that is, involving a state~$\states$, a control~$\controls$
% (composed of the volume of oil, free gas and water in the reservoir, the total pore
% volume and the reservoir pressure)
and an evolution function~$\dynamics$ such that, for each time step $\timeindex$, we have
$\states_{\timeindex+1} = \dynamics(\states_{\timeindex}, \controls_{\timeindex})$. It is
shown in Appendix~\ref{sect:Formulation_details} that a possible state - which is the one
we henceforth consider,
% which is the one we will consider in the sequel,
for modeling the reservoir when using the black-oil model and conservation laws for a
tank-like reservoir - is the 5-dimensional vector
$\states_{\timeindex}=\left( \totalreservoiroil_{\timeindex},
  \totalreservoirgas_{\timeindex}, \totalreservoirwater_{\timeindex},
  \porevolume_{\timeindex}, \respressure_{\timeindex} \right)$. Its components are defined
in Table~\ref{tab:state_definition}, where Sm$^3$ stands for standard cubic meter (the
volume taken by a fluid at standard pressure and temperature condition: 1.01325 Bara and
$15^{\circ}$C), and Bara stands for absolute pressure in Bar.
\begin{savenotes}
\begin{table}[htbp!]
  \centering
  \begin{tabular}{cl}
    \toprule
    Symbol                              & Definition                                    \\
    %%%
    \midrule
    $\totalreservoiroil_{\timeindex}$   & Amount of oil in the reservoir (Sm$^3$)
                                         at time $\timeindex$                 \\
    %%%
    $\totalreservoirgas_{\timeindex}$   & Amount of free gas in the reservoir (Sm$^3$) at
                                         time $\timeindex$                              \\
    %%%
    $\totalreservoirwater_{\timeindex}$ & Amount of water in the reservoir (Sm$^3$) at
                                         time $\timeindex$                              \\
    %%%
    $\porevolume_{\timeindex}$          & Total pore volume of the reservoir (m$^3$) at
                                         time $\timeindex$                              \\
    %%%
    $\respressure_{\timeindex}$         & Reservoir pressure (Bara) at
                                         time $\timeindex$                              \\
    \bottomrule
  \end{tabular}
  \caption{Definition of the components of the state\label{tab:state_definition}}
\end{table}
\end{savenotes}

More precisely, to obtain the evolution function~$\dynamics$ of the content of the
reservoir between time $\timeindex$ and $\timeindex+1$, we compute the amounts of fluids
(oil, gas, water) produced during the period $[\timeindex, \timeindex+1[$. We denote them
by
$\np{\oilproduced_{\timeindex}, \gasproduced_{\timeindex}, \waterproduced_{\timeindex}}$
and they are described in Table~\ref{tab:control_definition}. We obtain the production
values with a mapping
$\GeneralProductionFunction = \np{\GeneralProductionFunction^{(1)},
  \GeneralProductionFunction^{(2)}, \GeneralProductionFunction^{(3)}}: \XX \times \UU \to
\RR^3$ such that
$\left( \oilproduced_{\timeindex}, \gasproduced_{\timeindex},
  \waterproduced_{\timeindex}\right) = \GeneralProductionFunction(\states, \controls)$.
The production mapping~$\GeneralProductionFunction$ depends on the form and specifications
of the production network. We present two examples of such $\GeneralProductionFunction$ in
the numerical applications of Section~\ref{sect:numerical_applications}, with details in
Appendix~\ref{sect:Formulation_details}.
% \jpc{dire qu'on en verra deux examples dans la suite dans les deux applications}

\begin{table}[htbp!]
  \centering
  \begin{tabular}{cl}
    \toprule
    Symbol                      & Definition                               \\
    %%%
    \midrule
    $\oilproduced_{\timeindex}$   & Volume of oil produced (Sm$^3$) during
                                    $[\timeindex, \timeindex+1[$           \\
    %%% 
    $\gasproduced_{\timeindex}$   & Volume of gas produced ($Sm^3$) during
                                    $[\timeindex, \timeindex+1[ $          \\
    %%% 
    $\waterproduced_{\timeindex}$ & Volume of water produced ($Sm^3$) during
                                    $[\timeindex, \timeindex+1[$           \\
    %%% 
    \bottomrule
  \end{tabular}
  \caption{Definition of the productions\label{tab:control_definition}}
\end{table}

We make the following assumptions on the reservoir (as formulated in~\citet{Dake}): first,
the fluids contained in the reservoir follow a \emph{black-oil} model; second, we consider
that we have a tank-like reservoir. Thanks to those two standards assumptions, we can
formulate the reservoir and the production system as a controlled dynamical system.

\begin{proposition}
  There exists a function $\Xi: \XX \times \UU \rightarrow \RR$ such that the following
  function $\dynamics: \XX \times \UU \rightarrow \RR^{5}$
  \begin{equation}
    \dynamics: (\states, \controls)
    \mapsto
    \left( \begin{array}{l}
        \states^{(1)}-\GeneralProductionFunction^{(1)}(\states, \controls)\\
        \begin{multlined}
          \states^{(2)}-\GeneralProductionFunction^{(2)}(\states, \controls)
          +\Big[ \states^{(1)} \solutiongasoilratio(\states^{(5)}) \\[-3ex] %
          % [-3ex] Remove the added vertical space
          - \bp{\states^{(1)}-\GeneralProductionFunction^{(1)}(\states, \controls)}
          \solutiongasoilratio \bp{\Xi(\states, \controls)}
          \Big]
        \end{multlined}
        \\[4ex] % [4ex] Add a vertical space with the next line
        \states^{(3)}-\GeneralProductionFunction^{(3)}(\states, \controls)\\
        \states^{(4)} \left(1+\porecompressibility
          (\Xi(\states, \controls)-\states^{(5)})
        \right)\\
        \Xi(\states, \controls)
      \end{array}\right)
    \label{eq:dynamics_expression}
  \end{equation}
  is the dynamics of the reservoir in~\eqref{eq:dynamics_general} (with
  $x = (x^{(1)}, \dots, x^{(5)})$, $\solutiongasoilratio$ a given function of the
  reservoir pressure called the \emph{solution gas} function, and $\porecompressibility$ a
  given parameter called the \emph{pore compressibility} of the reservoir).
  \label{prop:dynamic_formulation}
\end{proposition}

\begin{proof}
  See Appendix~\ref{sect:Formulation_details}.
\end{proof}

\section{Two numerical applications}
\label{sect:numerical_applications}
We now present two numerical applications that illustrate how the material balance
formulation can be used. The numerical applications are done on simple reservoirs. In
\S\ref{subsect:2tanks_numerical_case}, the first application is a gas reservoir that can be
modeled with two tanks and with a connection, of known transmissivity, linking them
together. It illustrates how the formulation can be applied to complex cases with multiple
tanks. In \S\ref{subsect:oil_wi_numerical_case}, the second application we consider is an
oil reservoir where pressure is kept constant through water injection. This shows how we
can take into account injection to go beyond the first recovery of oil and gas. All
numerical applications were performed on a computer equipped with a Core i7-4700K and
16~GB of memory.

\subsection{A gas reservoir with one well}
\label{subsect:2tanks_numerical_case}
In the first application, we consider a real gas reservoir, for which production data are
available. The recorded data come from a field approaching abandonment. We only
considered a sub-field of a much larger field, the sub-field being constituted of an
isolated reservoir with one well.
% \cyrille{avant changement demander par Alejandro}
% The recorded data comes from a field approaching abandonment which is
% constituted of an isolated reservoir and one well and which is a sub-field of a larger
% field which is not considered here.

Our goal here is to show how simple cases can be tackled with the material balance
formulation, and that the formulation can also be applied to cases with multiple tanks. We
first present a state reduction of this case. We then present a model with one tank, and
then a model with two tanks, mimicking an evolutive construction of the reservoir model.
Indeed, when optimizing a real petroleum production system, the models are improved as
data are analyzed. Hence, reservoir models will get more complex to fit the gathered
exploitation data, such as going from a one tank model to a two tanks model. We therefore
present the models following such timetable, going from a cruder to a more refined
reservoir model.

\paragraph{Characteristics of the case.}
The geology of this particular sub-field makes it perfect for a tank model, as proved by
many years of perfectly matched production. Also, the simplicity of the fluids with a high
methane purity makes the black-oil model a very realistic assumption. The reservoir can be
modeled with either one or two tanks, while the well’s perforations are modeled with a
known stationary inflow performance relationship, noted
$\InflowPerformanceRelationship\gassuper$. The two tanks model is illustrated in
Figure~\ref{fig:two_coupled_reservoir_one_well}. We do not consider the rest of the
network, so that we will not have to take into account any vertical lift performance (VLP)
necessary to lift oil to the surface. This implies that the only control we consider is
the bottom hole flowing pressure (BHFP), $\presvar_{\timeindex}$, resulting in the problem
known as \emph{optimization at the bottom of the well}. We hence assume that there is no
``pipe'' necessary to move gas from the reservoir to the surface, thus assuming that the
network is only constituted of the well-perforations which allow the production of gas.
Indeed, optimizing with the bottom hole flowing pressure makes it easier to compare the
different reservoir models, as we directly act on the reservoir. Adding the vertical lift
performance only adds a layer of complexity to the comparison of the models, while the
only benefit would be to get results closer to an actual field production. All in all,
adding the vertical lift performance only adds more constraints on the mathematical
formulation and may mask the impact of the reservoir model. As the focus of this paper is
to present a formulation with a new reservoir model, we decided not to take into account
the vertical lift performance. We also did not try to go beyond the two tanks model.

% \jpc{Il faut definir ce que veut dire optimization at the bottom of the well pour le lecteur matheux}

\begin{figure}[ht!]
  \begin{center}
    \begin{tikzpicture}
      % coordinates of the well points
      \coordinate (well1) at (0, 0);
      \coordinate (well2) at (1, 0);
      \coordinate (well3) at (0.5, 0.86);

      % coordinates of the reservoir points (t: top, b:bottom, l:left, r:right)
      \coordinate (res_1_bl) at (1.5, 0);
      \coordinate (res_1_tr) at (2.5, 1);
      \coordinate (res_2_bl) at (3, 0);
      \coordinate (res_2_tr) at (4, 1);

      % draw the objects
      \draw[thick,black] (well1)--(well2)--(well3)--cycle;
      \draw[thick,black] (res_1_bl) rectangle (res_1_tr);
      \draw[thick,black] (res_2_bl) rectangle (res_2_tr);

      % coordinate for the connexions
      \coordinate (well_con) at (0.71, 0.5);
      \coordinate (res_1_w_con) at (1.5, 0.5);
      \coordinate (res_1_2_con) at (2.5, 0.5);
      \coordinate (res_2_1_con) at (3, 0.5);

      % draw the connexions
      \draw[thick, red] (well_con) -- (res_1_w_con);
      \draw[thick, red] (res_1_2_con) -- (res_2_1_con);

      % Add the text:
      \node (w) at (0, 1.5) {\noir{well’s perforations}};
      \node (res1) at (2, 2) {\noir{first tank}};
      \node (res2) at (4, 1.5) {\noir{second tank}};
      \draw[thick,black,->] (w)--++(0.5,-1);
      \draw[thick,black,->] (res1)--++(0,-1.5);
      \draw[thick,black,->] (res2)--++(-0.5,-1);
      % \draw[->] (well3)++(0,-1.66)--++(0,0.80);
      % \draw[->] (res_1_bl)++(0.5,-1.3)--++(0,1.3);
      % \draw[->] (res_2_bl)++(0.5,-0.8)--++(0,0.8);
      \node[align=center] (connexion_1t_2t) at (3.5,-1) {assumed \\ transmissivity};
      \draw[thick,black,->] (connexion_1t_2t)--(2.75, -0.25)--(2.75,0.5);
      \node (connexion_well_1t) at (0.5,-1) {$\InflowPerformanceRelationship\gassuper$};
      \draw[thick,black,->] (connexion_well_1t)--(1.2, -0.25)--(1.2,0.5);

    \end{tikzpicture}
    \caption{Representation of the two tanks model}
    \label{fig:two_coupled_reservoir_one_well}
  \end{center}
\end{figure}
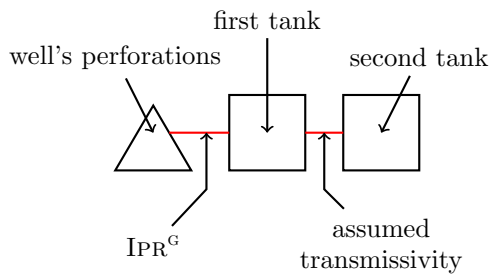

% \subsubsection{Formulation and state reduction}
% \label{subsubsect:gas_tank_formulation_and_state_reduction}
\paragraph{Formulation and state reduction.}
In this first application, we consider a reservoir that contains only gas and water. We
first assume that we only produce some gas, and that no fluids are re-injected in the
reservoir. Moreover, we assume that there is no water production, and thus the amount of
water remains stationary. Therefore, 
$\totalreservoirwater_{\timeindex}=\totalreservoirwater_{0}$ for all
$\timeindex \in \timeset$, the initial amount of water $\totalreservoirwater_{0}$ being
known. We therefore only need to consider the evolution of the amount of gas, the pressure
and the total pore volume as states variables. As shown in
Appendix~\ref{sect:state_reduction}, we can further reduce the state, and we only need to
consider the amount of gas in the reservoir as the reservoir state. Since we do an
optimization at the bottom of the well, we only have one control to consider, the
bottom-hole flowing pressure, noted $\presvar_{\timeindex}$. We therefore have state
$\states_{\timeindex} = \totalreservoirgas_{\timeindex}$ and control
$\controls_{\timeindex} = \presvar_{\timeindex}$.

The optimization problem we consider here is to maximize the revenue of the gas
production. At each time $\timeindex$, we sell gas at price $r_{\timeindex}$, with a
discount factor $\rho$. The general optimization problem~\eqref{eq:general_formulation}
after state and control reduction when considering the gas reservoir and one tank is given
by
\begin{subequations}
  \begin{align}
    \max_{(\totalreservoirgas_{\timeindex}, \presvar_{\timeindex},
    \respressure_{\timeindex}, \gasproduced_{\timeindex})}
    &
      \sum_{\timeindex=0}^{T-1} \rho^{\timeindex} r_{\timeindex} \gasproduced_{\timeindex}
      \label{eq:gas_tank_obj}
    \\
    s.t. ~~ &
      \totalreservoirgas_{0} = \states_{0}
      \eqfinv
    \\
    &
    \respressure_{\timeindex} = \Psi_{1\mathbf{T}}(\totalreservoirgas_{\timeindex})
    \eqsepv
    \forall \timeindex \in \timeset
    \eqfinv
    \label{eq:Psi-mapping}
    \\
    &
    \gasproduced_{\timeindex} = \frac{\InflowPerformanceRelationship\gassuper \left(
    \respressure_{\timeindex} - \presvar_{\timeindex} \right) }
    {\gasformationfactor(\respressure_{\timeindex})}
    \eqsepv
    \forall \timeindex \in \timeset \setminus \na{\horizon}
    \eqfinv \label{eq:IPR_production}
    \\
    &
    \totalreservoirgas_{\timeindex+1} = \totalreservoirgas_{\timeindex} -
    \gasproduced_{\timeindex} % + \reservoirflow_{\timeindex}
    \eqsepv
    \forall \timeindex \in \timeset \setminus \na{\horizon}
    \eqfinv
    \label{eq:dynamic_gas_reservoir}
    \\
    &
    \gasproduced_{\timeindex} \geq 0
    \eqsepv
    \forall \timeindex \in \timeset \setminus \na{\horizon}
    \eqfinv \label{eq:positive_production}
    \\
    &
    \totalreservoirgas_{\timeindex} \geq 0
    \eqsepv
    \forall \timeindex \in \timeset
    \eqfinv
    \\
    &
    \presvar_{\timeindex} \geq 0
    \eqsepv
    \forall  \timeindex \in \timeset \setminus \na{\horizon}
    \eqfinv
  \end{align}
  \label{eq:formulation_gas_reservoir_one_tank}
\end{subequations}
as detailed in Appendix~\ref{sect:state_reduction}.

\subsubsection{One tank gas reservoir model}
\label{subsubsect:one_tank_gas}

\paragraph{Fitting model to real data.}
We use production data from a sector of a real gas field, to check that the reservoir
model described with the Constraints~\eqref{eq:Psi-mapping}
and~\eqref{eq:dynamic_gas_reservoir} accurately follows real measurements on the gas field
after fitting the model. More precisely, we apply a given real production schedule on a
part of the field (only one well), and check that the pressure we simulate in the
reservoir is close to the corresponding measured pressure. The historical production spans
over 15 years, and we have monthly values, which is why we consider monthly time steps for
Problem~\eqref{eq:formulation_gas_reservoir_one_tank}.

As can be seen in Figure~\ref{fig:sim_1_res_1_well}, the one tank model fits the
observation. However, there is a gap between the simulated and measured pressures
whose relative value may exceed $10\%$. Since the simulated pressure tends to be higher
on the first half of the
production, we start by underestimating the decline of the production. Then, during the
second half of the production, the simulated pressure is lower than the measured pressure,
which means we overestimate the decline of the production. This elastic effect is most
likely due to the simplification of removing the secondary tank in the model. Indeed, the
secondary tank act as a buffer which reacts slowly, explaining the extra pressure at the
beginning and then sustaining a better value of the pressure later on.

\begin{figure}[htbp!]
  \centering
  \includegraphics[scale=0.5]{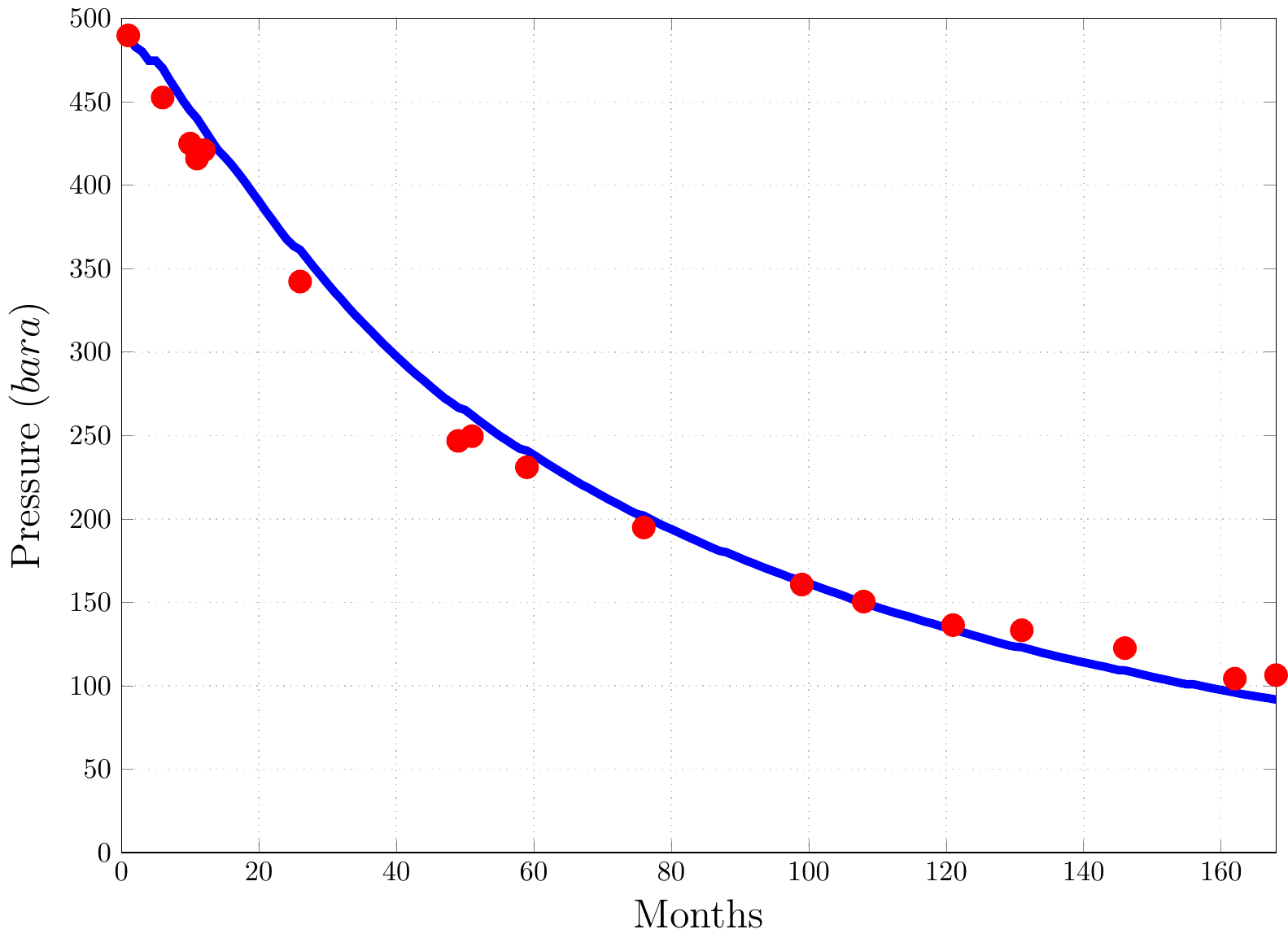} % for sc
  \caption{Comparison of the simulated one tank reservoir pressure to the historical
    measured pressure when applying the same (historical) production schedule. The blue
    curve is the simulated pressure in the tank, whereas the red dots are the measured
    pressures.}
  \label{fig:sim_1_res_1_well}
\end{figure}

\paragraph{Optimization of the production on the one tank approximation.}
We use dynamic programming (see Algorithm~\ref{alg:dynamic_programming}) to get an optimal
production policy. We consider that the revenue per volume of gas is the historical gas
spot price of TTF (Netherlands gas market) from 2006 to 2020, and we do not consider any
operational cost.

We now present the results of the one tank model. The results are illustrated in
Figures~\ref{fig:states_traj_1tank} and~\ref{fig:production_1tank}, and summarized in
Table~\ref{tab:one_tank_summary}. We notably remark in Figure~\ref{fig:production_1tank}
that the optimal production stops when prices are low as we fully take advantage of the
perfect knowledge of the future prices.

% \jpc{Je comprends pas cette longue discussion, on ne sait pas d'ou proviennent les
% productions qui sont des donn\'{e}es on ne sait pas quel \'{e}tait le crit\`{e}re optimis\'{e} donc \c{c}a sert pas a
% grand chose de discuter aussi longuement}

There is a massive increase in the total gains when using the optimal policy, compared to
the real production. We also produce far more over the optimization time period
(\numprint[MSm^3]{2850} instead of \numprint[MSm^3]{2250}). However, those results are not
truly comparable. We do not have access to the criteria used to choose the real
production. Optimized and real productions cannot be compared as they do not share the
same objective function. Moreover, since the considered case is a small part of a much
larger production network, we cannot compare the results to the actual production policy
used for fitting the model, which was made with the rest of the network in mind.
Furthermore, our optimization is made at the bottom of the well (BHFP).
We only take into
account the inflow performance of the well, not the vertical lift necessary to bring the
gas to the surface.
The resulting rates are therefore not fully realistic, reaching values
closer to a multi-well development. Finally, the historical production was made without
knowing future prices, and could also have been made with other constraints to ensure a
minimal production of the field, or having a positive cash-flow (constraints due to the
field's exploitation contract). While not directly comparable, this gas reservoir
application still illustrates one of the best-case scenarios of the dynamic programming
approach, and it shows how much could be gained from using a multistage material balance
formulation.

% The numerical experiments also reveal that the value function seems to almost be an
% affine function, that grows with the initial volume of in place gas in the reservoir.
Since the dynamic programming algorithm uses a discretization of the state space $\XX_d$
and the control space $\UU_d$, we tried different uniform discretizations for the states
and controls spaces to prevent any side effects due to the chosen discretization. We do
not observe notable changes in the value function past a \numprint{10000} points uniform
discretization of the state space and a 20 points discretization of the control space,
which are the values we used in this case study. Details on the effect of the
discretization can be found in Appendix~\ref{sect:discretization_impact}.

\begin{figure}[htbp!]
  \centering
  \includegraphics[width=0.45\textwidth]{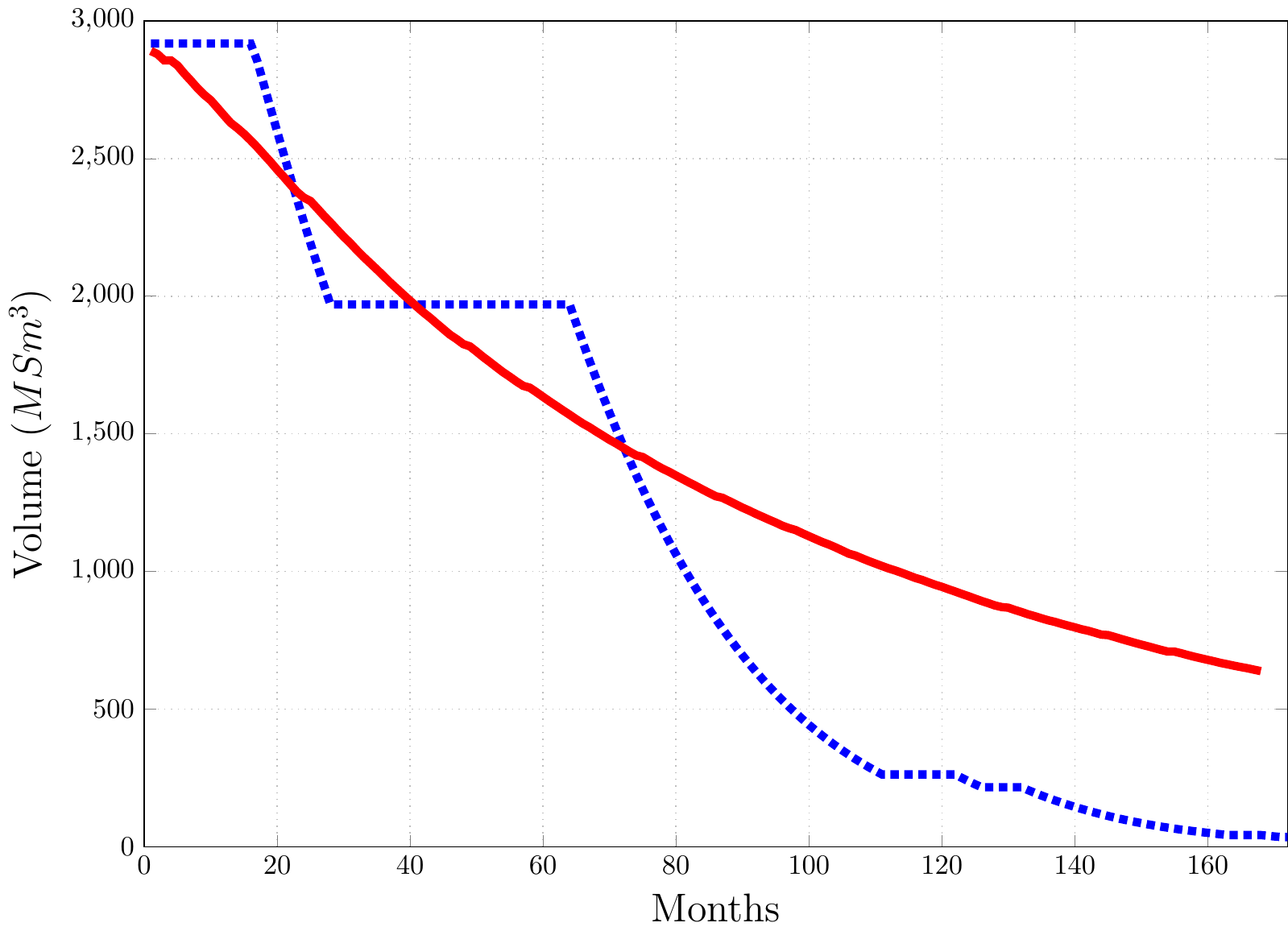}
  \caption{Evolution of the content of the reservoir in the one tank model. The doted blue
    curve is the optimal (anticipative) trajectory of the amount of gas, while the red
    curve is the trajectory with the historical production.}
  \label{fig:states_traj_1tank}
\end{figure}

\begin{figure}[ht!]
  \centering
  \includegraphics[width=0.45\textwidth]{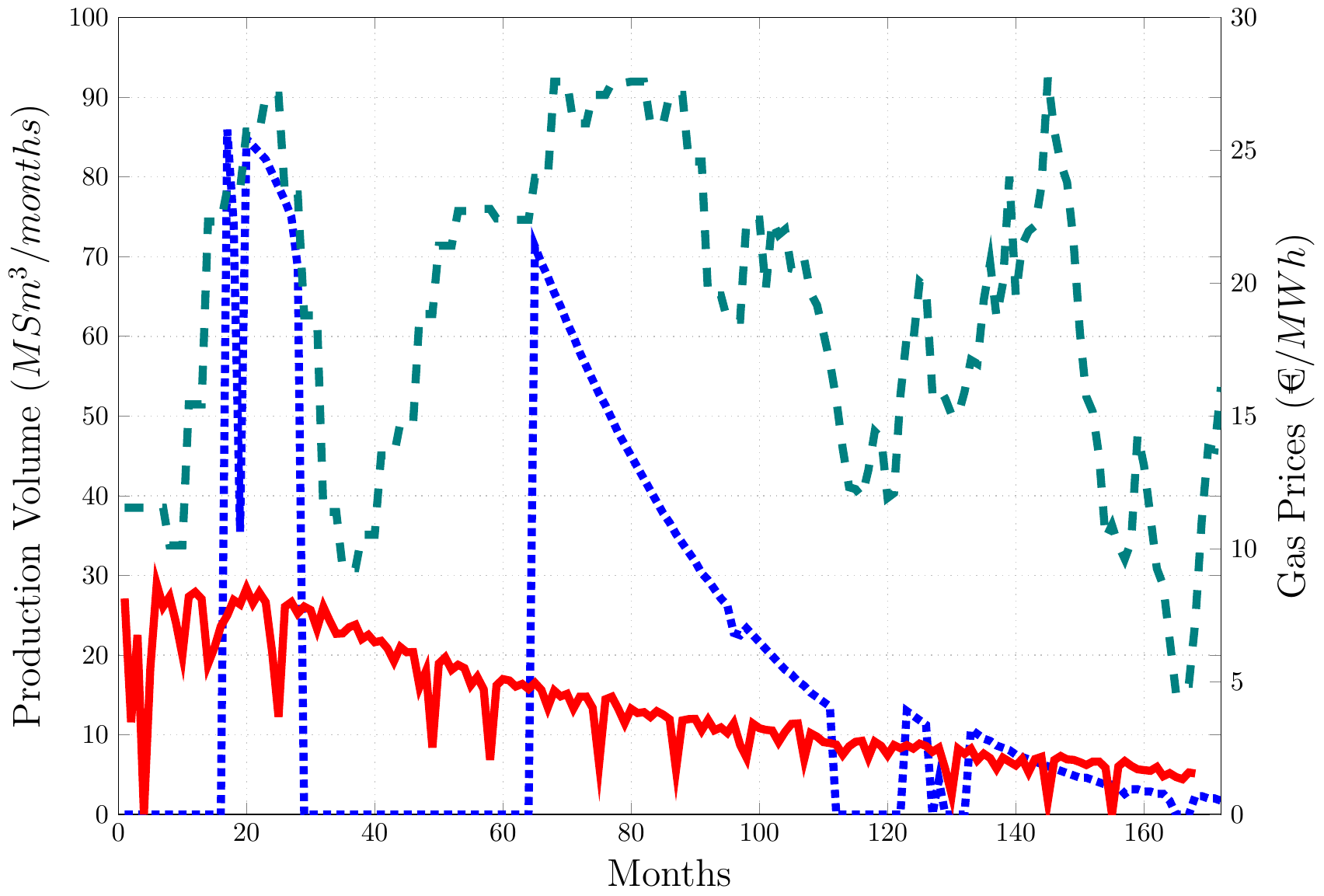}
  \caption{Trajectories of the production. The dotted blue curve is the optimal
    (anticipative) production in the one tank model, the red one is the historical
    production, whereas the dashed green curve is the average monthly gas price.}
  \label{fig:production_1tank}
\end{figure}

\paragraph{Comparison to policy derived from decline curves.}
In this paragraph, we compare the material balance formulation to those using decline
curves or oil-deliverability curves, such as
in~\cite{Grossmann98,GUPTA2012,GUPTA2014,Marmier}. The decline curves formulation and the
way to numerically obtain decline curves are given in
Appendix~\ref{sect:decline_curve_explanation}. The following proposition shows that the
decline curves formulation is equivalent to the material balance formulation when
considering a one-tank model.

\begin{proposition}
  The formulation using decline curves, written
  \begin{subequations}
    \begin{align}
      \max_{\controls} ~ ~
      &
        \sum_{\timeindex=0}^{\horizon}
        \rho^{\timeindex} \costfunct_{\timeindex}(\controls_{\timeindex})
        \label{eq:obj_dc}
      \\
      s.t. ~~
      &
        \oilproduced_{\timeindex} \leq \declinecurve \left( \sum_{s = 0}^{\timeindex-1}
        \oilproduced_{\cumulatedtimeindex} \right)
        \eqsepv \forall \timeindex \in \timeset \setminus \{0\}
        \label{eq:production_dc}\\
      & \controls_{\timeindex} \in \admcontrolset_{\timeindex}\left(
        \sum_{s = 0}^{\timeindex-1} \oilproduced_{\cumulatedtimeindex}
        \right)
        \eqsepv \forall \timeindex \in \timeset \eqfinv
        \label{eq:controls_admissibility_dc}
    \end{align}
    \label{eq:general_formulation_decline_curves}
  \end{subequations}
  is equivalent to the material balance formulation when the state of the reservoir is
  one-dimensional.
  % (as in the optimization
  % problem~\eqref{eq:formulation_gas_reservoir_one_tank}).
  \label{prop:equivalance_dc_mbal_1state}
\end{proposition}
\begin{proof}
  See Appendix~\ref{sect:decline_curve_explanation}
\end{proof}

We obtain the decline curve $\declinecurve$ used in Inequality~\eqref{eq:production_dc} by
first computing the maximal production value for the same discrete states as the ones used
in the dynamic programming approach. Then, piecewise interpolation between the computed
values is used to obtain the value of the decline curve everywhere. It is worth noting
that, when using piecewise linear approximation for the decline curves, the maximization
problem~\eqref{eq:general_formulation_decline_curves} turns out to be a MIP (Mixed Integer
Problem) with linear constraints and with more than \numprint{170000} binary variables. We
solve that MIP by using the commercial solver Gurobi 9.1. The results are given in
Table~\ref{tab:one_tank_summary}. Since the material balance
formulation~\eqref{eq:formulation_gas_reservoir_one_tank} uses a one-dimensional state, we
obtain similar results between the material balance formulation and the formulation using
a decline curve in accordance with Proposition~\ref{prop:equivalance_dc_mbal_1state}. The
two approaches thus yield similar production policies. Note however that the dynamic
programming approach has a lower computation time than a naive implementation of the
decline curve formulation. One could decrease the precision on the decline curve
formulation, by using fewer points to describe the decline curve. This would improve its
computation time.
% and could have a negligible impact on the value of the optimization if
% the remaining points are correctly chosen \_pc{trop vague \c{c}a veut dire quoi correctly
%   chosen ?}.
As this is not the focus of this paper, we did not do such refinement of the
numerical experiments for the decline curve formulation.

\begin{table}[htbp!]
  \begin{center}
    \begin{tabular}[bct]{crr}
      \toprule
                            &  CPU time (s)  & Value (M\euro) \\
      \midrule
      Material Balance      &     \numprint{653}        &   743          \\
      Decline Curves        &     \numprint{3882}       &   743          \\
      \bottomrule
    \end{tabular}
    \caption{Comparison with regards to CPU time and value between the material balance
      and decline curve formulation for one tank\label{tab:one_tank_summary}}
  \end{center}
\end{table}

\subsubsection{Two tanks gas reservoir model}

\label{subsubsect:two_tanks_gas}
\paragraph{Fitting data.}
We check if the fitted two tanks reservoir model accurately follows real measurement on
the gas field. We use the same data as in the one tank case. The two tanks model more
accurately fits the observations, as is depicted in Figure~\ref{fig:sim_2_res_1_well_2}
(we have a gap of less than $5\%$ for each measured point). Since the two tanks model is
closer to the observations, we consider that it is the reference of ``truth'' when
comparing results of the one tank approximation and the two tanks model.

\begin{figure}[htbp!]
  \begin{center}
    \includegraphics[scale=0.5]{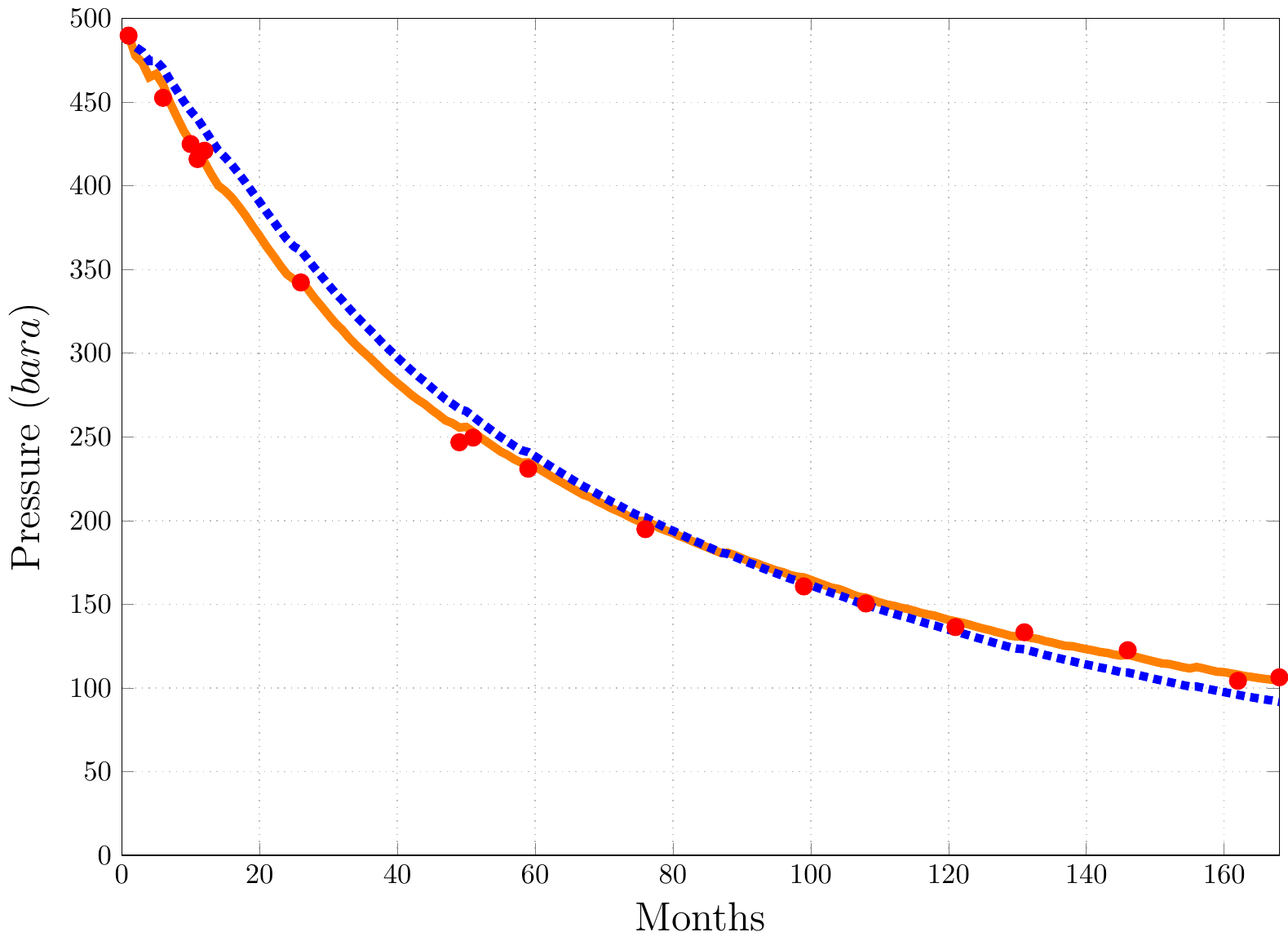} % for sc
  \end{center}
  \caption{Comparison of the simulated two tanks reservoir pressure to the measured
    pressure when applying the same production schedule. The blue dotted curve recalls the
    pressure obtained using the one tank model. The orange continuous curve is the
    pressure in the first tank obtained using the two tanks model. The red dots are the
    measured pressure at the bottom of the well.}
  \label{fig:sim_2_res_1_well_2}
\end{figure}

\paragraph{Optimal production with two tanks.}

We now present the results of the two tanks model. The only changes compared to the one
tank model are on the states and on the dynamics of the reservoir. We use the same prices,
and, again, we only do an optimization at the bottom of the well (BHFP). Details on the
obtained optimal controls and states trajectory are given in
Figure~\ref{fig:states_traj_2tanks} and Figure~\ref{fig:production_2tanks}. Once again, we
observe that production stops when prices are low, benefiting fully from anticipating the
future prices. We also note that more ``pauses'' are present in the productions when
compared to the one tank model (four instead of three). The ``pauses'' allow the second
tank to replenish the first one (see Figure~\ref{fig:states_traj_2tanks}). Indeed,
production resumes at months 50 to 60, before stopping again for five months. We can then
observe that the amount of gas in the first tank is replenished, before we resume
production at month 65, at the same date as in the one tank model. We end up producing
some more gas than with the one tank model ($\numprint[Sm^3]{2900}$ instead of
$\numprint[Sm^3]{2850}$).

\begin{figure}[htbp!]
  \centering
  \includegraphics[width=0.45\textwidth]{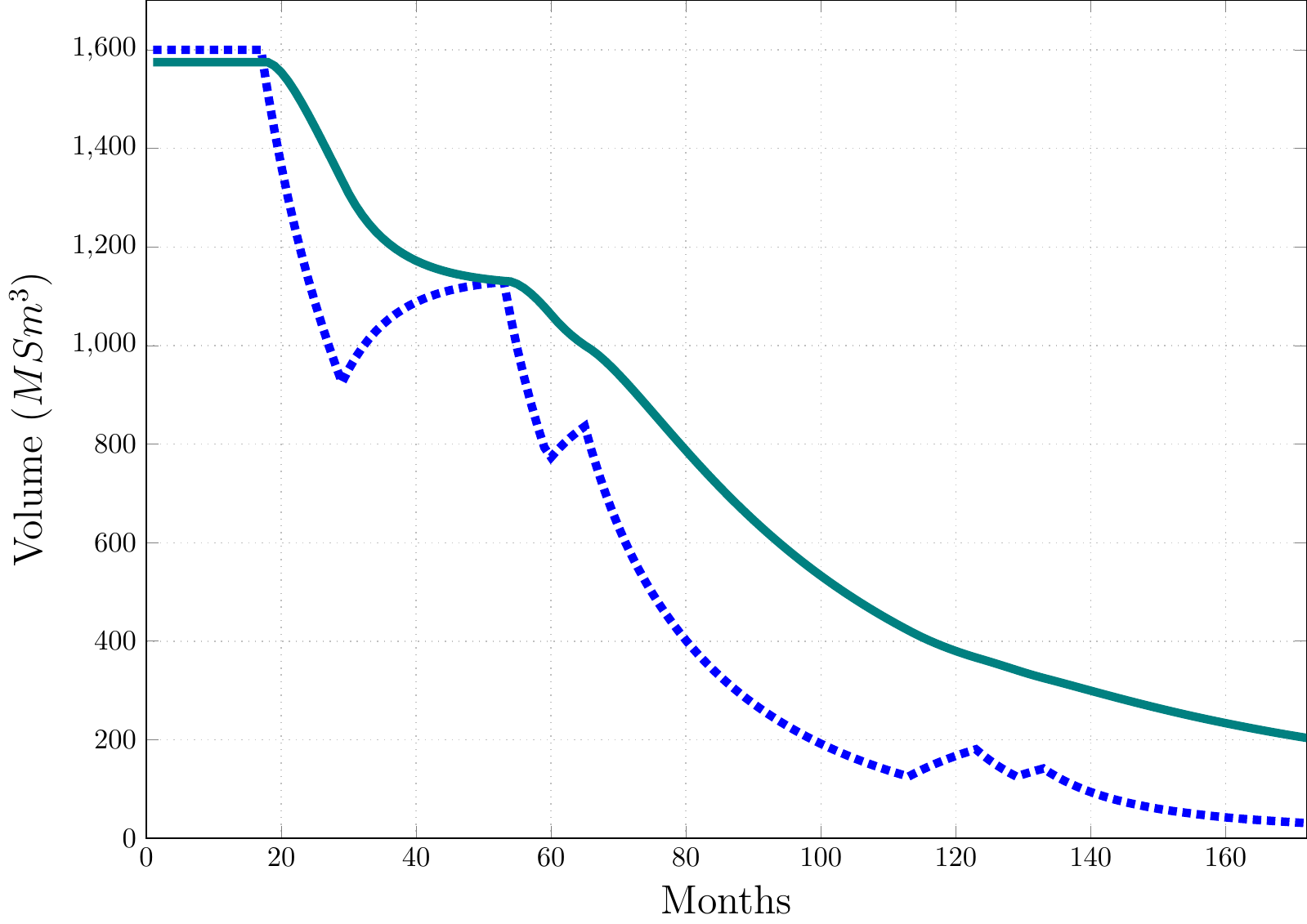}
  \caption{Evolution of the content of the reservoirs when applying the optimal
    (anticipative) policy in the two tanks model. The dotted blue curve shows the content
    of the first tank (linked to the well) while the green curve shows the content of the
    second tank.}
  \label{fig:states_traj_2tanks}
\end{figure}

\begin{figure}[htbp!]
  \centering
  \includegraphics[width=0.45\textwidth]{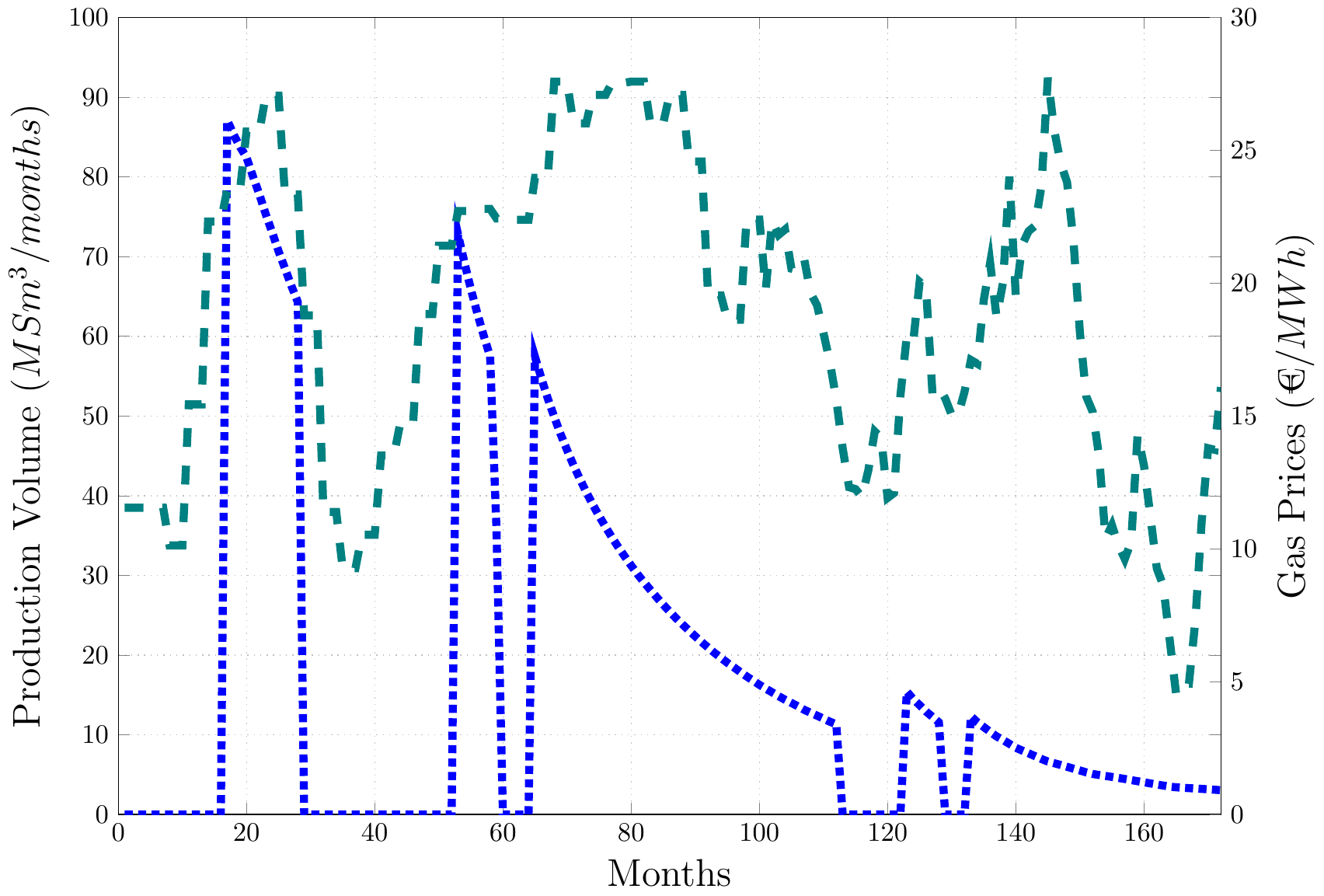}
  \caption{Trajectory of the optimal production in the two tanks model. The dotted blue
    curve is the optimal (anticipative) production, whereas the dashed green curve is the monthly gas
    price.}
  \label{fig:production_2tanks}
\end{figure}

We tried different discretizations for the state space. Notably, using more than $400$
possible states per tank and $10$ possible controls did not yield any significant
improvement in the computed value function. Details on the impact of the discretization
are given in Appendix~\ref{sect:discretization_impact}.

Numerical experiments also reveal that the initial value function $\bellval_0$ is almost
an affine function of the sum of the states. This seems to imply that the one tank and two
tanks model should yield similar results. Such a statement does not hold true, as
confirmed by the numerical experiments described in the next paragraph.
% Indeed, if the
% value function truly depended exclusively on the sum of the states, the optimal control
% would also be a function of the sum of the states.

\paragraph{Comparing the one tank formulation to the two tanks formulation.}
To compare the results between the two tanks and one tank formulations, we consider that
the two tanks material balance model is the reference. A given sequence of controls
$\np{\controls_{\timeindex}}_{\timeindex \in \timeset\setminus\na{\horizon}}$ admissible for the one
tank model is not necessarily admissible for the two tanks model. Indeed, the admissible
control set is given by
$\admcontrolset(\states_{\timeindex}) = [0, \Psi_{1\mathbf{T}}(\states_{\timeindex})]$ for
the one tank model and by
$[0, \Psi_{2\mathbf{T}}^{(1)}(\states^{(1)}_{\timeindex})]$ for the two tanks model (see Appendix~\ref{subsect:state_reduction_gas}).

Thus, given a sequence of controls
$\np{\controls_{\timeindex}}_{\timeindex \in \timeset\setminus\na{\horizon}}$ admissible
for the one tank model, we produce an admissible sequence of controls for the two tanks
model with the use of a \emph{projection}
$\Pi_{1\mathbf{T} \to 2\mathbf{T}}: \UU^{\horizon} \times \XX \to \UU^{\horizon}$ given as
follows. The sequence
$\np{\widetilde{\controls}_{\timeindex}}_{\timeindex \in
  \timeset\setminus\na{\horizon}}=\Pi_{1\mathbf{T} \to 2\mathbf{T}}
\bp{\np{\controls_{\timeindex}}_{\timeindex \in \timeset\setminus\na{\horizon}},
  \states_0}$ is computed recursively for all 
$\timeindex \in \timeset\setminus\na{\horizon}$ by 
$
\widetilde{\controls}_{\timeindex} = \min\ba{\controls_{\timeindex},
  \Psi_{2\mathbf{T}}^{(1)}(\widetilde{\states}^{(1)}_{\timeindex})}
% eqfinv
$, where $\widetilde{\state}_{\timeindex}$ is defined at time $0$ by
$\widetilde{\states}_{0} = \states_0$, and for all $\timeindex > 0$ by
\( \widetilde{\states}_{\timeindex+1} =
\dynamics_{2\mathbf{T}}(\widetilde{\states}_{\timeindex},
\widetilde{\controls}_{\timeindex}) \). We can get a sequence of admissible controls for
the two tanks model by applying the projection $\Pi_{1\mathbf{T} \to 2\mathbf{T}}$ on a
sequence of admissible controls for the one tank model.

% Since an admissible sequence of controls
% $\np{\controls_{\timeindex}}_{\timeindex \in \timeset\setminus\na{\horizon}}$ for the one
% tank model satisfies the condition $\controls_{\timeindex} \geq 0$, for all $\timeindex$,
% it only needs to satisfy the condition
% $\controls_{\timeindex} \leq \Psi_{2\mathbf{T}}^{(1)}(\states_{\timeindex})$ (where
% $\states_{\timeindex}$ is given by the recursion
% $\states_{\timeindex'+1} = \dynamics(\states_{\timeindex'}, \controls_{\timeindex'})$) to
% be admissible in the two tanks model. By construction, the projection
% $\Pi_{1\mathbf{T} \to 2\mathbf{T}} \bp{\np{ \controls_{\timeindex}}_{\timeindex \in
%     \timeset\setminus\na{\horizon}}, \states_0}$ satisfies that condition. Hence, we can
% get a sequence of admissible controls for the two tanks model by applying the projection
% $\Pi_{1\mathbf{T} \to 2\mathbf{T}}$ on an sequence of admissible controls for the one tank
% model.

To compare the one tank and two tank models, we project the optimal sequence of controls
returned by the dynamic programming algorithm on the one tank formulation thanks to the
projection $\Pi_{1\mathbf{T} \to 2\mathbf{T}}$. As can be seen in
Figure~\ref{fig:production_1tank_corrected}, the projected sequence of controls differs
from the non-projected sequence: the dotted curve, which represents the projected
sequence, is below the dashed curve, which represents the optimal sequence for the one
tank model.

As depicted in Figures~\ref{fig:production_1tank_corrected}
and~\ref{fig:gains_1tank_corrected}, the production planning given by the one tank
optimization problem differs from the production planning given by the two tanks
optimization problem. Moreover, the production planning of the one tank model gives lower
gains than anticipated, and is worse than the optimal two tanks model planning. The one
tank optimization is thus optimistic on the optimal value of the problem when applied with
the reference model. Moreover, there is a $5\%$ difference in value between the one tank
and two tanks models (a value of \numprint[M\text{\euro}]{703} for the translated one tank
production planning against \numprint[M\text{\euro}]{736} for the two tanks production
planning). This discrepancy illustrates how having a more accurate model of the reservoir
can have a substantial impact on the optimal planning, all other things being equal. It
also shows that, contrarily to the assumption presented at the end of the previous
paragraph (that the two models could yield similar results if the value function only
depended on the sum of the states), the optimal value and control cannot be found with a
one tank approximation, and the optimal controls and value functions are not functions of
the sum of the states.

\begin{figure}[ht!]
  \centering
  \includegraphics[width=0.45\textwidth]{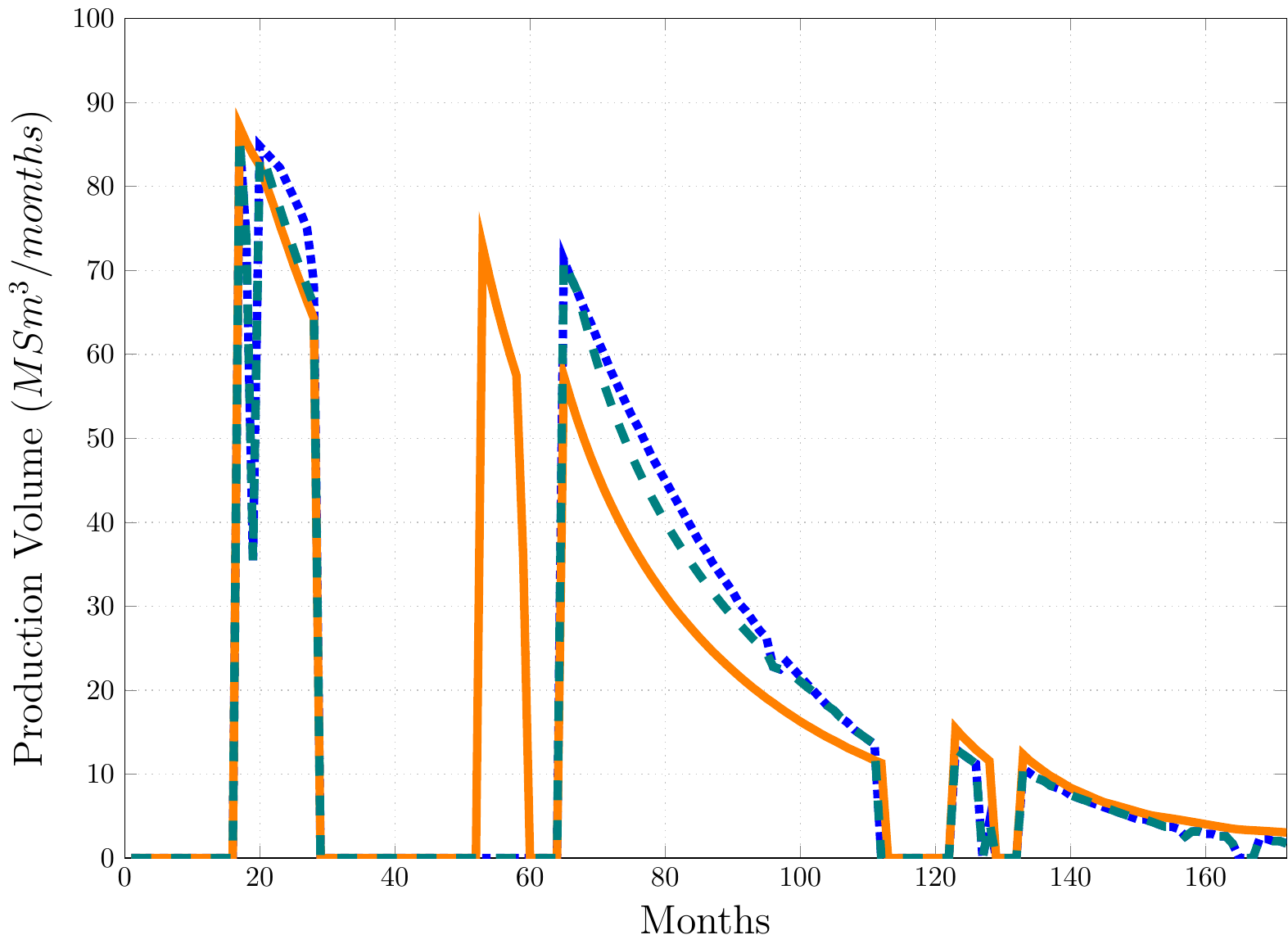}
  \caption{Comparison of the trajectory of the production with the two tanks model as
    reference. The dotted blue curve is the production planning in the one tank model,
    the orange curve is for the two tanks model. The dashed green curve is the
    production planning of the one tank model projected in the two tanks model}
  \label{fig:production_1tank_corrected}
\end{figure}

\begin{figure}
  \centering
  \includegraphics[width=0.45\textwidth]{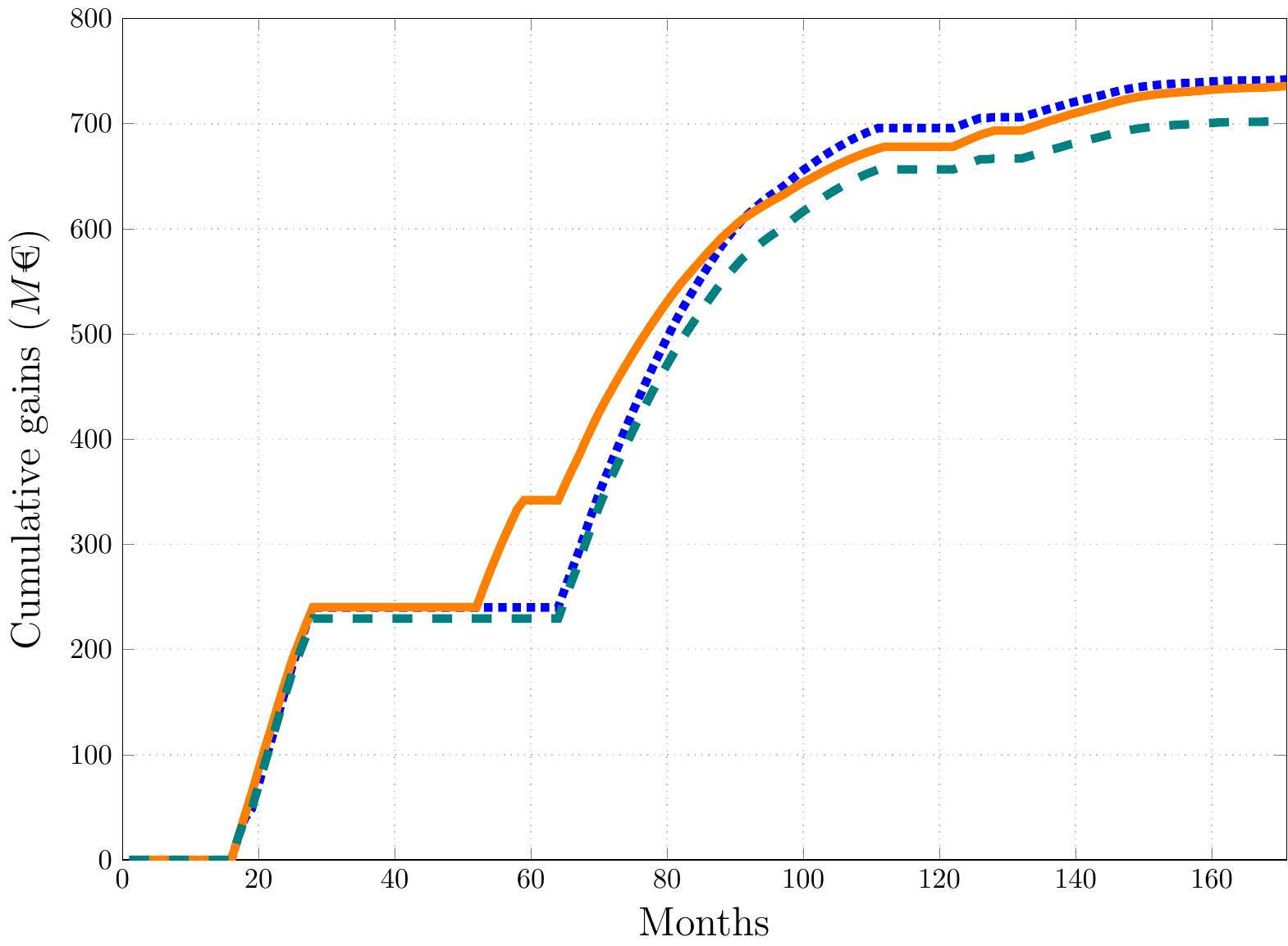}
  \caption{Cumulated gains with the two tanks model as reference. The dotted blue curve
    is the cumulated gains of the one tank planning in the one tank model, the orange
    curve is the cumulated gains of the two tanks planning in the two tanks model, and the
    dashed green curve is the cumulated of the one tank planning projected in the two
    tanks model}
  \label{fig:gains_1tank_corrected}
\end{figure}

\paragraph{Comparison to decline curves with two tanks.}
We have numerically compared the decline curve and the material balance formulations in a
context where they are known to be equivalent, that is, the one tank formulation. We now
produce numerical experiments in a context where the equivalence is not assured: two tanks
connected with a known transmissibility. We have generated decline curves for the two
tanks formulation by following the procedure described in
Appendix~\ref{sect:decline_curve_explanation}.
% As with the comparison between the one tank
% and the two tanks models, we consider that the two tanks model is the reference.
The
results returned by the decline curve formulation provide an admissible production in the
two tanks model, as it is constrained by an admissible production schedule. We can
therefore directly compare the results obtained by the decline curves approach and the two
tanks model. The results of the
optimization of the two formulations are compiled in
Table~\ref{tab:DC_to_DP_comparison_two_tanks}. We end up having close results, with a
difference in optimal values of $0.7\%$, but with a large difference in computing times.
However, it appeared that such close results were due to the selected price scenario.
Using different prices by randomizing the order in which the different prices appear, the
gap between the two approaches widens from $0.5\%$ up to $4\%$. This implies that the
initial price considered was an almost best-case scenario for the decline curves approach.
It also shows that the decline curves approach is far less robust to changes in the price
data, and that it cannot benefit as efficiently as the material balance formulation of
some effects of the two tanks dynamical system, such as waiting for the second tank to
empty itself into the first one.

\begin{table}[htbp!]
  \begin{center}
    \begin{tabular}[bct]{ccc}
      \toprule
                            &  CPU time (s)  & Value (M\euro) \\
      \midrule
      Material Balance      &     \numprint{706}        &   \numprint{736}          \\
      Decline Curves        &     \numprint{7825}       &   \numprint{731}          \\
      \bottomrule
    \end{tabular}
  \end{center}
  \caption{Comparison with regards to CPU time and value between the material balance
    and decline curve formulation for two tanks with the initial prices sequence.
    \label{tab:DC_to_DP_comparison_two_tanks}}
\end{table}

Overall, this application suggests that the material balance approach can work on complex
cases, and that dynamic programming is well suited to optimize an oil field. Moreover,
there can be differences with results from the decline curves approach, which are likely
to grow larger with the complexity of the system.

\subsection{An oil reservoir with water injection}
\label{subsect:oil_wi_numerical_case}
The second application is an oil reservoir with water injection. The goal is to
demonstrate how the formulation can be used beyond primary recovery cases, on a
numerically simple case. We consider that we have one reservoir which contains both oil
and water, produced under pressure maintenance by water injection. Moreover, we consider
that the initial pressure is above the bubble-point, which eliminates the possibility of
having free gas in the reservoir. This allows us to have once again a one-dimensional
state: either the water (which we used for the numerical applications), or the oil in the
reservoir. We have $\states_{\timeindex} = \totalreservoirwater_{\timeindex}$ and
$\controls=\presvar_{\timeindex}$. Here, we want to maximize the revenue of the oil
production. The optimization problem~\eqref{eq:general_formulation} now becomes
\begin{subequations}
  \begin{align}
    \max_{(\totalreservoirwater_{\timeindex},
    \presvar_{\timeindex}, \watercut_{\timeindex})} ~~
    &
      \sum_{\timeindex=0}^{\horizon-1}
      \bigg( \rho^{\timeindex}
      r_{\timeindex} \alpha \frac{\respressure - \presvar_{\timeindex}}
      {\oilformationfactor(\respressure)}
      \np{1-\watercut_{\timeindex}} \nonumber \\
    &
      \qquad
      - \rho^{\timeindex} c_{\timeindex} \alpha \frac{\respressure -
      \presvar_{\timeindex}} {\waterformationfactor(\respressure)}
      \bigg)
    \label{eq:obj_water_injection}\\
    % &
    %   \begin{aligned}
    %     \sum_{\timeindex=0}^{\horizon-1}
    %     \bigg( \rho^{\timeindex}
    %     r_{\timeindex} \alpha \frac{\respressure - \presvar_{\timeindex}}
    %     {\oilformationfactor(\respressure)}
    %     \np{1-\watercut_{\timeindex}} \\
    %     - \rho^{\timeindex} c_{\timeindex} \alpha \frac{\respressure -
    %       \presvar_{\timeindex}} {\waterformationfactor(\respressure)}
    %     \bigg)
    %   \end{aligned}
    % \label{eq:obj_water_injection}\\
    s.t. ~~
    & \watercut_{\timeindex} = \watercutfunct\left(\frac{\totalreservoirwater_{\timeindex}
      \waterformationfactor(\respressure)}{\porevolume}\right)
      \eqsepv \forall \timeindex \in \timeset
      \eqfinv\label{eq:wct_water_injection}\\
    & \totalreservoirwater_{\timeindex+1} = \totalreservoirwater_{\timeindex} -
      \alpha \frac{\respressure - \presvar_{\timeindex}}
      {\waterformationfactor(\respressure)}\np{\watercut_{\timeindex}-1}
      \eqsepv \forall \timeindex \in \timeset
      \eqfinv\label{eq:dynamics_water_injection}\\
    & \waterproduced_{min} \leq
      \alpha \frac{\respressure - \presvar_{\timeindex}}
      {\waterformationfactor(\respressure)}\np{\watercut_{\timeindex} - 1}
      \leq \waterproduced_{max} \eqsepv \forall  \timeindex \in \timeset
      \eqfinv\label{eq:water_production_bound_water_injection}\\
    & \oilproduced_{min} \leq
      \alpha \frac{\respressure - \presvar_{\timeindex}}
      {\oilformationfactor(\respressure)}\np{1-\watercut_{\timeindex}}
      \leq \oilproduced_{max} \eqsepv \forall \timeindex \in \timeset
      \eqfinv\label{eq:oil_production_bound_water_injection}\\
    & \presvar_{\timeindex} \geq 0 \eqsepv \forall \timeindex \in \timeset
      \eqfinp\label{eq:positive_pressure}
  \end{align}
  \label{eq:formulation_water_injection}
\end{subequations}

The objective function (Equation~\eqref{eq:obj_water_injection}) is divided in two
components. At time $\timeindex$, we consider a discount factor $\rho$ and the
price $r_{\timeindex}$ of the oil, whereas injecting water costs $c_{\timeindex}$ per
cubic meter. The revenue is hence
\[
  \sum_{\timeindex=0}^{\horizon-1}
  \rho^{\timeindex}
  \Big( 
  r_{\timeindex} \oilproduced_{\timeindex} -
  c_{\timeindex} \waterinjected_{\timeindex}
  \Big)
  \eqfinp
\]
Replacing the produced oil $\oilproduced_{\timeindex}$ and the injected water
$\waterinjected_{\timeindex}$ by the relevant functions of the controls (see
Equations~\eqref{eq:def_Fo_WI}
and~\eqref{eq:def_Fwi_function}) leads to the objective
function~\eqref{eq:obj_water_injection}).

We assume that the water-cut function $\watercutfunct$ (the amount of water produced when
extracting one cubic meter of liquid at standard conditions) is given by a piecewise
linear function. The water-cut depends on the water saturation $\watersaturation$
(proportion of water in the reservoir pore volume). Since the reservoir pressure is kept
constant, the total pore volume is constant and the water saturation expression is thus
$\watersaturation_{\timeindex} = \frac{\totalreservoirwater_{\timeindex}
  \waterformationfactor(\respressure)}{\porevolume}$. This gives us
constraint~\eqref{eq:wct_water_injection}.

Since we want to keep a constant pressure in the reservoir, we need to re-inject enough
water to replace the extracted oil. Replacing the oil with water gives a new dynamic for
$\totalreservoirwater_{\timeindex}$ as in Equation~\eqref{eq:dynamics_water_injection}.
Constraints~\eqref{eq:water_production_bound_water_injection}
and~\eqref{eq:oil_production_bound_water_injection} details the oil and water produced
depending on the control~$\presvar_{\timeindex}$ with their respective bounds. The details
of the formulation are given in Appendix~\ref{sect:state_reduction}.

We do a monthly optimization, with the historical Brent prices for years 2000--2020 as the
prices in the objective function~\eqref{eq:obj_water_injection}, and a water injection
cost of \numprint[\text{\euro}/m^{3}]{1}. Details on the resulting trajectory of the
content of the reservoir can be found in Figure~\ref{fig:states_traj_oil_wi}, whereas
details on the production can be found in Figure~\ref{fig:production_oil_wi}. As
previously discussed in \S\ref{subsect:2tanks_numerical_case}, the optimal policy yields
more production when prices are high, and stops producing when they are low. The
production goes from one bound to the other (zero production, with
$\presvar_{\timeindex} = \respressure$, and full production, with
$\presvar_{\timeindex} = 0$).

The production also does not fully deplete the reservoir, which means that it is not
advantageous to completely deplete the reservoir if one wants to maximize the profit over
the optimization time frame (there is still \numprint[MSm^{3}]{18.2} of oil in the
reservoir at time $\horizon$, as can be seen in Figure~\ref{fig:states_traj_oil_wi}).
Indeed, production slowly diminishes with the volume of oil
$\totalreservoiroil_{\timeindex}$ in the reservoir, as can be seen in
Figure~\ref{fig:production_oil_wi}. It is more advantageous to wait for high prices
instead of producing, as it would reduce the possible future production. This leads to
halting production with some reserves still in the reservoir, as we prefer to wait for a
higher price instead of producing when prices are low. As a side effect, numerical
experiments reveal that the initial value function $\bellval_0$ is almost linear with
regards to the state $\states_0$. However, we only considered simple constraints on the
production. As more constraints will be added to the problem, other behaviors will
certainly appear. CPU time was \numprint[s]{1575} for a $\numprint{100000}$ discretization
of the state variable, with a value of \numprint[M\text{\euro}]{3376}. Impact of the
discretization can be found in Appendix~\ref{sect:discretization_impact}.

\begin{figure}[htbp!]
  \centering
  \includegraphics[width=0.45\textwidth]{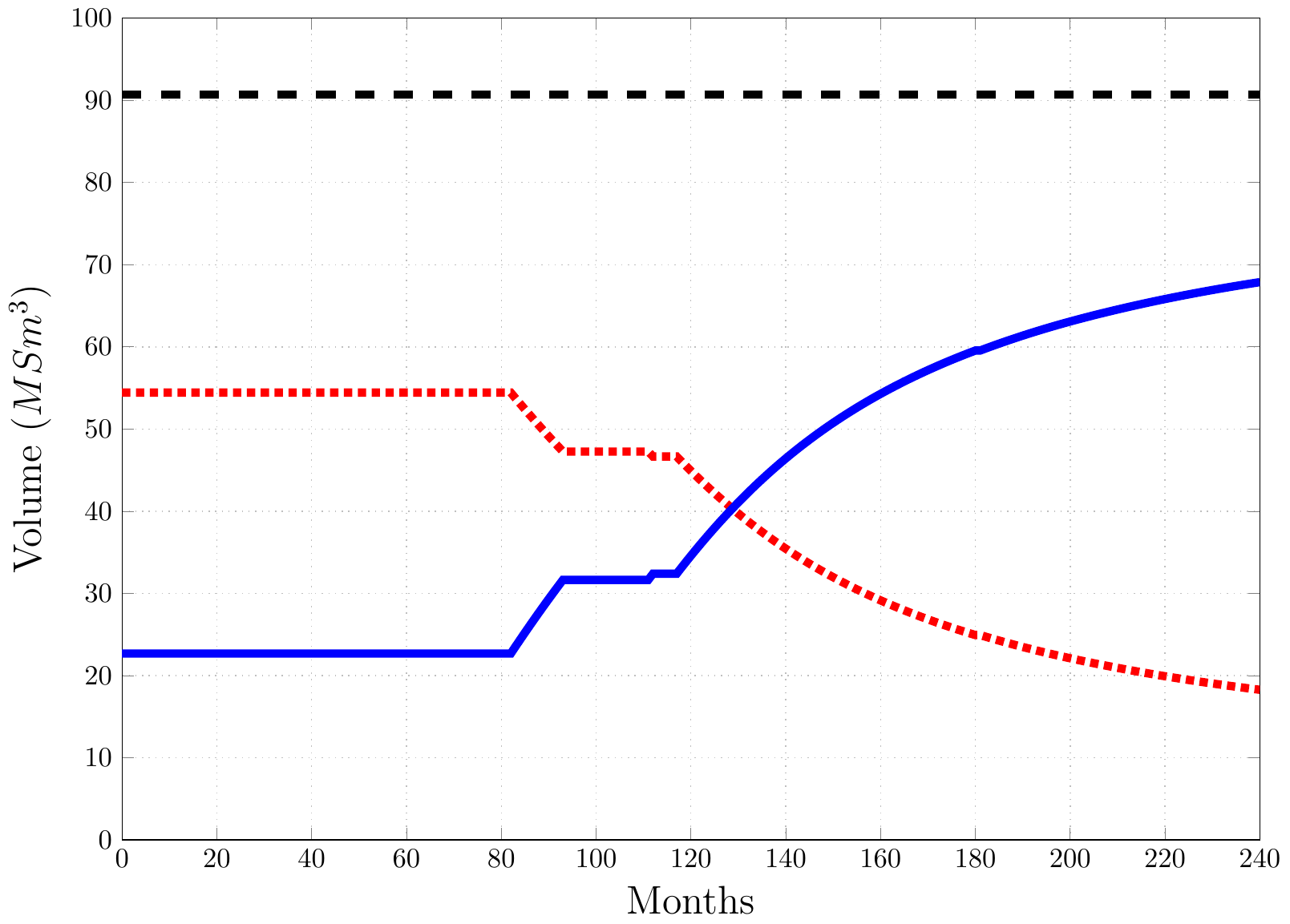}
  \caption{Evolution of the content of the reservoir when applying the optimal policy in
    the oil reservoir model. The blue curve shows the volume of water in the reservoir,
    whereas the dotted red curve is the volume of oil the reservoir. The dashed black
    curve represents the total pore volume}
  \label{fig:states_traj_oil_wi}
\end{figure}

\begin{figure}[htbp!]
  \centering
  \includegraphics[width=0.45\textwidth]{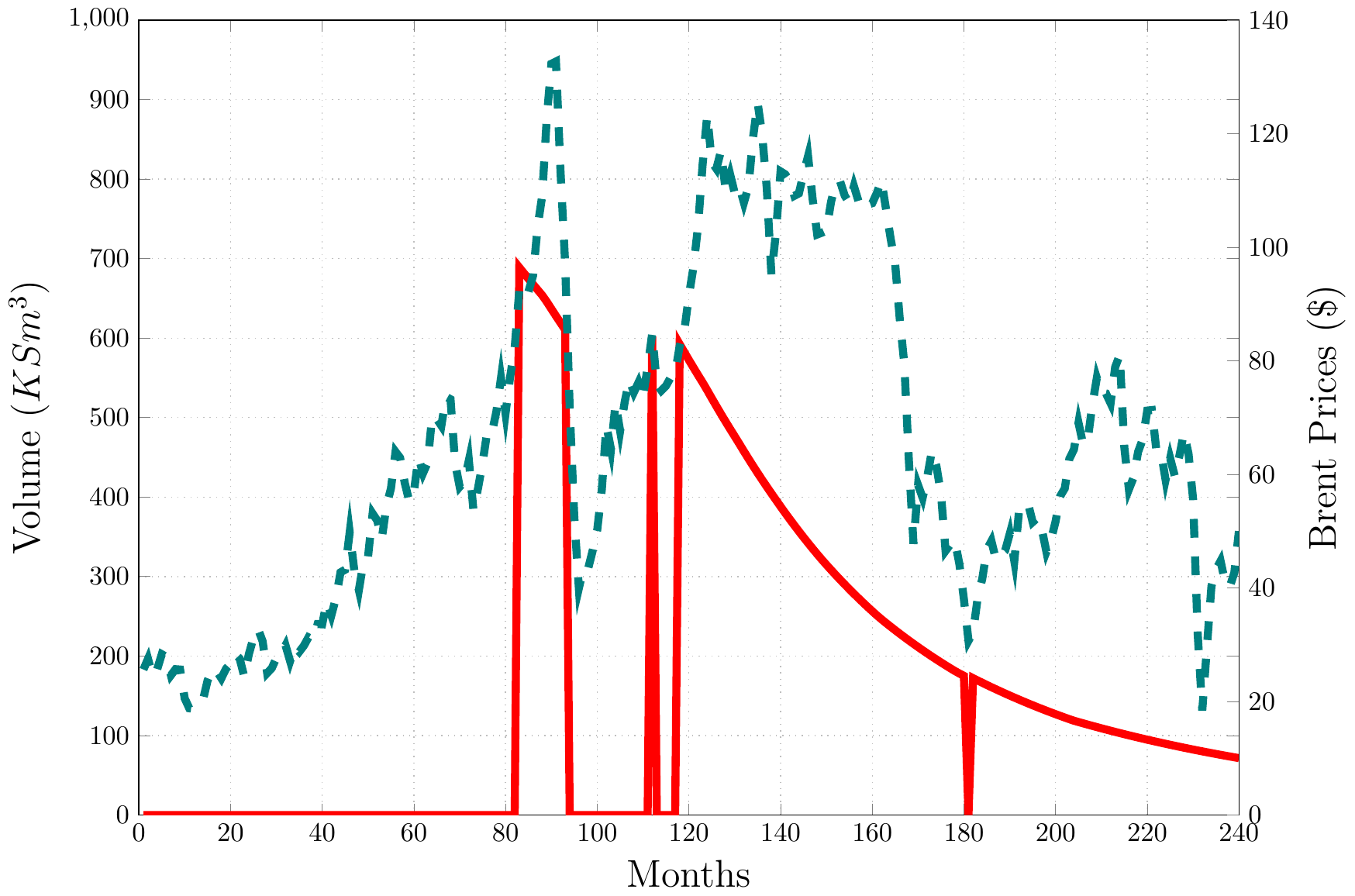}
  \caption{Trajectory of the optimal production in the oil reservoir model. The red curve
    is the optimal production, whereas the dashed green curve is the monthly oil price}
  \label{fig:production_oil_wi}
\end{figure}

Overall, this application shows how we can apply the material balance approach beyond first
recovery of oil and gas, and that it can be used on different kinds of reservoirs.

\section{Conclusion}
\label{sect:conclusion}
% \bleue{
  In this paper, we have presented a mathematical formulation for the optimal
  management over time of an oil production network as a multistage optimization problem.
  In this formulation, the reservoir is modeled as a controlled (non-linear) dynamical
  system derived from material balance equations and the black-oil model. The state of the
  derived dynamic system is of dimension five, which is quite large for numerical
  resolution via dynamic programming algorithm. However, we were able to use Dynamic
  Programming to numerically solve the management optimization problem for specific cases
  of interest with either oil or gas, both presenting a reduced dimensionality of the
  state. We have also shown that our mathematical formulation is an improvement over
  decline curves formulation. First, as predicted by the theory, we replicated results
  from decline curve formulations when considering the first recovery of a one tank system
  (as seen in \S\ref{subsubsect:one_tank_gas}). Second, in more complex cases with
  inter-connected tanks, as described in \S\ref{subsubsect:two_tanks_gas}, we have shown
  that we can surpass the NPV returned by the decline curve formulation. Third, we have
  gone beyond the first recovery of hydrocarbons, as we have shown in
  \S\ref{subsect:oil_wi_numerical_case}, where we took into account water injection.
  Finally, it is to be noted that the dynamic programming algorithm can be used in a
  stochastic framework. As an example, we could add uncertainties to the oil and gas
  prices, instead of assuming that they are known in advance and thus deterministic.
  % This
  % will render the optimization process more realistic, as an optimal production policy is
  % highly dependent on current prices and on their future evolution.
% \jpc{rajouter le cote temps et futur price evolution which is not known}
  Moreover, an even more realistic formulation with \emph{partial observation} of the
  content of the reservoir could also be explored. Indeed, in oil production systems, the
  initial state of the reservoir is not known. Such a formulation is amenable to dynamic
  programming, as will be explored in future works. % }

\section*{Acknowledgements} We would like to thank TEPNL in general and Erik Hornstra in
particular for providing data used in this paper.

\appendix
\section{Detailed construction of the reservoir as a dynamical system}
\label{sect:Formulation_details}
% \rouge{Mettre les pr\'{e}cisions sur le r\'{e}servoir ici plut\^{o}t qu'en section
% \ref{sect:Reservoir_Dynamical_System}?}
In this section, we detail the construction of the reservoir as a dynamical system. This
serves as the proof of Proposition~\ref{prop:dynamic_formulation}.

\subsection{Constitutive equations assuming the black-oil model for the fluids}
\label{sub:black-oil}
The black-oil model relies on the assumption that there are at most three \emph{fluids} in
the reservoir: oil, gas and water. Moreover, the fluids can be present in the reservoir in
up to two phases: a liquid phase, and a gaseous phase. A black-oil representation of a
reservoir can be seen in Figure~\ref{fig:reservoir_representation}. The three fluids, oil,
gas and water, can be present in the liquid phase and the gas in the liquid phase is
denoted as \emph{dissolved gas}. By contrast, it is assumed that in the gaseous phase,
only gas, denoted as \emph{free gas}, can be present.
% This can be seen \jpc{pas vraiment en fait on ne voit pas les 4 composants!} in
% Figure~\ref{fig:reservoir_representation}, which is a representation of a reservoir in the
% black oil-model.

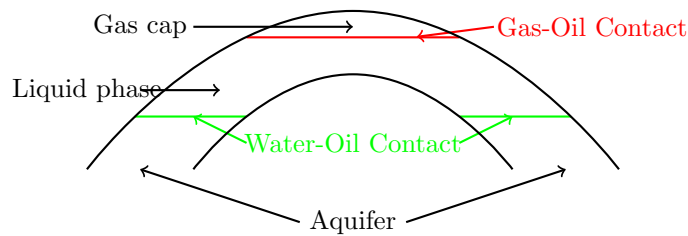
\begin{figure}[ht!]
  \begin{center}
    \begin{tikzpicture}[scale=0.7]
        \draw[red][thick] (-2,2.5)--(2,2.5);
        \draw[green][thick] (-2,1)--(-4.1,1);
        \draw[->][thick] (-3,2.7)--(0,2.7);
        \node at (0,0.5) {\verte{Water-Oil Contact}};
        \draw[green,->][thick] (2, 0.5)--(3,1);
        \draw[green,->][thick] (-2,0.5)--(-3,1);
        \draw[red,->][thick] (2.65, 2.7)--(1.2,2.51);
        \node at (4.5,2.7) {\rouge{Gas-Oil Contact}};
        \node at (-5, 1.5) {Liquid phase};
        \draw[->][thick] (-4,1.5)--(-2.5,1.5);
        \draw[green][thick] (2,1)--(4.1,1);
        \node at (-4, 2.7) {Gas cap};
        \draw[thick] plot [smooth,tension=1] coordinates{(-5,0) (0,3) (5, 0)};
        \draw[thick] plot [smooth,tension=1] coordinates{(-3,0) (0,1.8) (3, 0)};
        \node at (0,-1) {Aquifer};
        \draw[->][thick] (1,-1)--(4,0);
        \draw[->][thick] (-1,-1)--(-4,0);
    \end{tikzpicture}
    \caption{Black-oil Representation of a reservoir}
    \label{fig:reservoir_representation}
  \end{center}
\end{figure}

Therefore, in the black-oil model, we consider the following four components
\begin{itemize}
\item $\totalreservoiroil$, the standard volume of oil in the liquid phase,
\item $\totalreservoirgas$, the standard volume of free gas in the gaseous phase,
\item $\dissolvedgas$, the standard volume of dissolved gas in the liquid phase,
\item $\totalreservoirwater$, the standard volume of water in the liquid phase,
\end{itemize}
where \emph{standard volume} is the volume taken by a fluid at standard pressure and
temperature condition (1.01325 Bar and $15^{\circ}$C), also known as stock tank
conditions. The units of standard volumes are preceded by a capital S, as in Sm$^3$ for
standard cubic meter.

There are functions in the black-oil model to convert standard volumes into in situ
volumes in the reservoir under a given pressure and temperature. The set of functions
describing the pressure, volume and temperature behavior of the fluids, under the
black-oil assumption, is call the PVT (Pressure-Volume-Temperature) model. We consider
here a simplified black-oil model, assuming that the temperature in the reservoir is
stationary and uniform, which is a common assumption for a geological formation such as a
reservoir. There are four PVT functions, one per component, which are given in
Table~\ref{tab:PVT_Definition}. The PVT functions only depend on the reservoir pressure
under the stationary and uniform temperature assumption. As an example, given the oil
standard volume, $\totalreservoiroil$, and the reservoir pressure, $\respressure$, the oil
volume in the reservoir is given by
$\totalreservoiroil \times \oilformationfactor(\respressure)$.

\begin{table*}[ht!]
    \begin{tabular}{cp{12cm}}
        \toprule
        Notations               & Description \\
        \midrule
        $\oilformationfactor$   &
        Oil formation volume factor. It is the volume in barrels occupied in
        the reservoir, at the prevailing pressure and temperature, by one stock
        tank barrel of oil plus its dissolved gas. (unit: $rb/stb$)  \\
        $\gasformationfactor$   &
        Gas formation volume factor. It is the volume in barrels that one
        standard cubic foot of gas will occupy as free gas in the reservoir
        at the prevailing reservoir pressure and temperature. (unit: $rb/scf$) \\
        $\waterformationfactor$ &
        Water formation factor. It is the volume occupied in the reservoir by one
        stock tank barrel of water. (unit: $rb/stb$) \\
        $\solutiongasoilratio$  &
        Solution (or dissolved) gas. It is the number of standard cubic
        feet of gas which will dissolve in one stock tank barrel of oil when both
        are taken down to the reservoir at the prevailing reservoir pressure and
        temperature. (unit: $scf/stb$) \\
        \bottomrule
    \end{tabular}
    \caption{Definition of the PVT functions\label{tab:PVT_Definition}}
\end{table*}
% \begin{tabular}{|c|p{12cm}|}
%   \hline
%   Notations & Description \\
%     %
%   \hline
%   $\oilformationfactor$ &
%   Oil formation volume factor. It is the volume in barrels occupied in
%   the reservoir, at the prevailing pressure and temperature, by one stock
%   tank barrel of oil plus its dissolved gas. (unit: $rb/stb$)  \\
%     %
%   \hline
%   $\gasformationfactor$ &
%   Gas formation volume factor. It is the volume in barrels that one
%   standard cubic foot of gas will occupy as free gas in the reservoir
%   at the prevailing reservoir pressure and temperature. (unit: $rb/scf$) \\
%     %
%   \hline
%   $\waterformationfactor$ &
%   Water formation factor. It is the volume occupied in the reservoir by one
%   stock tank barrel of water. (unit: $rb/stb$) \\
%     %
%   \hline
%   $\solutiongasoilratio$ &
%   Solution (or dissolved) gas oil ratio. It is the number of standard cubic
%   feet of gas which will dissolve in one stock tank barrel of oil when both
%   are taken down to the reservoir at the prevailing reservoir pressure and
%   temperature. (unit: $scf/stb$) \\
%     %
%   \hline
% \end{tabular}

One key characteristic of the black-oil model that we use is due to~\cite[chap
2]{danesh}, which states that the sum of the physical volumes in the reservoir associated
with the three components $\totalreservoiroil$, $\totalreservoirgas$,
$\totalreservoirwater$ is a decreasing function
\begin{equation}
  \respressure\mapsto
  \totalreservoiroil \times
  \oilformationfactor(\respressure) +
  \totalreservoirgas \times
  \gasformationfactor(\respressure) +
  \totalreservoirwater \times
  \waterformationfactor(\respressure)
  \eqfinv
  \label{eq:volume_equality}
\end{equation}
of the reservoir pressure.

The last characteristic of the black-oil model concerns the dissolved gas in the oil
$\dissolvedgas$. It is assumed in~\cite{Dake} that the standard volume of the dissolved
gas $\dissolvedgas$ is a function of both the standard volume of oil,
$\totalreservoiroil$, and the reservoir pressure, $\respressure$, as follows
\begin{equation}
  \dissolvedgas = \delta(\totalreservoiroil, \respressure)
  = \totalreservoiroil \times \solutiongasoilratio(\respressure)
  \eqfinp
  \label{eq:dissolved_gas_bis}
\end{equation}

\subsection{Conservation law in the reservoir}
\label{sub:conservation-law}
We assume that the reservoir structural integrity is guaranteed, so there is no leakage of
any fluids at any time. We can therefore write mass conservation equations, which are also
named \emph{material balance} equations in the oil literature, for each of the four
components introduced in \S\ref{sub:black-oil}. In order to write the material balance
equations of the reservoir, we need to consider the production volumes, $\oilproduced$,
$\gasproduced$ and $\waterproduced$ which are the standard volumes of oil, free gas and
water extracted from the reservoir.

Using material balance for the standard volume of oil in the liquid phase, we get
\begin{equation}
  \totalreservoiroil_{\timeindex+1} = \totalreservoiroil_{\timeindex} -
  \oilproduced_{\timeindex}
  ~~ \forall \timeindex \in \timeset\setminus\na{\horizon}
  \eqfinv
  \label{eq:MBalOil_bis}
\end{equation}
and, for the standard volume of water, we get
\begin{equation}
  \totalreservoirwater_{\timeindex+1} = \totalreservoirwater_{\timeindex} -
  \waterproduced_{\timeindex}
  ~~ \forall \timeindex \in \timeset\setminus\na{\horizon}
  \eqfinp
  \label{eq:MBalWater_bis}
\end{equation}

The material balance for gas requires some more developments as it mixes the standard
volume of free gas and the standard volume of dissolved gas. As given in
\S\ref{sub:black-oil}, at any time, $\timeindex \in \timeset$, the standard volume of
dissolved gas in the liquid phase $\dissolvedgas_\timeindex$ is given by
Equation~\eqref{eq:dissolved_gas_bis}. Therefore, between time $\timeindex$ and time
$\timeindex+1$, the standard volume of dissolved gas evolves from
$\dissolvedgas_{\timeindex} = \delta(\totalreservoiroil_{\timeindex},
\respressure_{\timeindex})$ to
$\dissolvedgas_{\timeindex+1} = \delta(\totalreservoiroil_{\timeindex+1},
\respressure_{\timeindex+1})$. Hence, the quantity
$(\dissolvedgas_{\timeindex} - \dissolvedgas_{\timeindex+1})$ of \emph{liberated} gas must
be added to the free gas material balance equation. Thus, for all
$\timeindex \in \timeset\setminus\na{\horizon}$, we obtain the following mass conservation
equation for the standard volume of free gas
\begin{align}
  \totalreservoirgas_{\timeindex+1}
  &= \totalreservoirgas_{\timeindex} - \gasproduced_{\timeindex}
    + (\dissolvedgas_{\timeindex} -
    \dissolvedgas_{\timeindex+1})
  \nonumber \\
  &=\totalreservoirgas_{\timeindex} - \gasproduced_{\timeindex}
    + \bp{\totalreservoiroil_{\timeindex} \times
    \solutiongasoilratio(\respressure_{\timeindex})
    - \totalreservoiroil_{\timeindex+1} \times
    \solutiongasoilratio(\respressure_{\timeindex+1})}
    \tag{by~\eqref{eq:dissolved_gas_bis}}
  \\
  &=\totalreservoirgas_{\timeindex} - \gasproduced_{\timeindex}
    + \Bp{\totalreservoiroil_{\timeindex}
    \times \solutiongasoilratio(\respressure_{\timeindex})
    - \bp{ \totalreservoiroil_{\timeindex} - \oilproduced_{\timeindex}}
    \times \solutiongasoilratio(\respressure_{\timeindex+1})}
  \tag{by~\eqref{eq:MBalOil_bis}}
  \\
  & = \totalreservoirgas_{\timeindex} - \gasproduced_{\timeindex}
    + \Big(\totalreservoiroil_{\timeindex} \times
    \bp{\solutiongasoilratio(\respressure_{\timeindex}) -
    \solutiongasoilratio(\respressure_{\timeindex+1})}
  \nonumber \\
  & \hspace{3cm}
    + \oilproduced_{\timeindex} \cdot
    \solutiongasoilratio(\respressure_{\timeindex+1})
    \Big)
    \eqfinp
    \label{eq:MBalGas_bis}
\end{align}

The last conservation equation is given by a physical volume constraint coming from the
fact that all four components of the reservoir are kept in the pores of the reservoir
rocks. We note $\porevolume$ the total pore volume of the reservoir. Following~\cite{Dake}
and assuming that the pore compressibility~$\porecompressibility$ is constant, the total
pore volume is a function of the pressure in the reservoir given by
% \begin{equation}
%   \left( \frac{\partial \porevolume}{\partial \respressure} \right)_{T} =
%   \porecompressibility \porevolume
%   \eqfinp
%   \label{eq:PoreCompression_exact}
% \end{equation}
% By integrating Equation~\eqref{eq:PoreCompression_exact} along $\respressure$, we can then
% express the total pore volume as a function of the pressure
\begin{equation}
  \label{eq:PoreVolume_exact}
  \porevolume_\timeindex = \reservoirvolume_0
  \mathrm{exp}\np{\porecompressibility \respressure_\timeindex}
  \eqsepv
  \forall \timeindex \in \timeset
  \eqfinv
\end{equation}
with $\reservoirvolume_0$ the asymptotic reservoir volume when pressure tends to $0$.

A linearized version of Equation~\eqref{eq:PoreVolume_exact} proposed in~\cite{Dake} is
\begin{equation}
  \frac{\porevolume_{\timeindex+1}-
    \porevolume_{\timeindex}}{\porevolume_{\timeindex}}
  = \porecompressibility \left(\respressure_{\timeindex+1} -
    \respressure_{\timeindex} \right)
  \eqsepv \forall \timeindex \in \timeset\setminus\na{\horizon}
  \eqfinv
  \label{eq:PoreCompression_bis}
\end{equation}
and is used to derive the state dynamics of the reservoir.

Now, we consider the \emph{saturations} of the fluids which are the proportions of the
available pore volume taken by each of the three fluids in the reservoir. Denoting by
$\oilsaturation$, $\gassaturation$ and $\watersaturation$ the saturations of respectively
the oil, free gas and water components, we obtain that the sum of the three saturations
must be equal to one over time
\begin{equation}
  \oilsaturation_{\timeindex} + \gassaturation_{\timeindex} +
  \watersaturation_{\timeindex} = 1 \eqsepv \forall \timeindex \in \timeset
  \eqfinp
  \label{eq:saturation_conservation_bis}
\end{equation}
Since, for all $\timeindex \in \timeset$ and
$i \in \na{\textsc{o}, \textsc{g}, \textsc{w}}$, we have that
\[
  S_\timeindex^{i} =
  \frac{V_t^i \times B_i(\respressure_{\timeindex})}{\porevolume_{\timeindex}}
  \eqfinv
\]
Equation~\eqref{eq:saturation_conservation_bis} gives
\begin{equation}
  \totalreservoiroil_{\timeindex} \times
  \oilformationfactor(\respressure_{\timeindex}) +
  \totalreservoirgas_{\timeindex} \times
  \gasformationfactor(\respressure_{\timeindex}) +
  \totalreservoirwater_{\timeindex} \times
  \waterformationfactor(\respressure_{\timeindex}) =
  \porevolume_{\timeindex}
  \eqfinv
  \forall \timeindex \in \timeset
  \eqfinp
  \label{eq:volume_equality_bis}
\end{equation}

\subsection{Construction of a production function}
\label{subsect:construction_of_GeneralProductionFunction}
The time evolution of the reservoir is driven by the three production volumes,
$\oilproduced$, $\gasproduced$ and $\waterproduced$ which are the standard volumes of oil,
free gas and water extracted from the reservoir.

Thus, the three production volumes may appear as possible controls on the reservoir.
However, when adding a production network to the reservoir model, the controls to be
considered are no longer production volumes, but decisions made upon
the production network, such as opening or closing a pipe, choosing the well-head or
bottom hole pressure, etc.

In the general case, we can assume that the physical model of the production network leads
to a production function~$\GeneralProductionFunction: \XX \times \UU \to \RR^3$, which
relates the production volumes to the variables of the reservoir
$\states = \np{\totalreservoiroil, \totalreservoirgas, \totalreservoirwater, \porevolume,
  \respressure}$ (we will show that $\states$ is a possible state of the reservoir) and to
the network controls~$\controls$, giving
\begin{equation}
  \left( \oilproduced, \gasproduced, \waterproduced  \right) =
  \GeneralProductionFunction(\states, \controls)
  \eqfinp
  \label{eq:general_production_definition}
\end{equation}

When considering only one well, a common assumption is that the production volumes are
given by the Inflow Performance Relationship $\InflowPerformanceRelationship$, which is a
function of the reservoir pressure $\respressure$, the bottom-hole pressure $\presvar$,
the saturation of water $\watersaturation$ and the saturation of gas $\gassaturation$.
More precisely, we obtain, for a one well model, that
\begin{align*}
  F_{\timeindex}^i =
    \GeneralProductionFunction^{i} \np{\states, \controls}
  &=
    \frac{
    \InflowPerformanceRelationship^i\np{\respressure_{\timeindex} -
    \presvar_{\timeindex}, \watersaturation_{\timeindex}, \gassaturation_{\timeindex} }}%
    { B_{i} \np{\respressure_{\timeindex}} }
    \eqsepv
    \forall i \in \na{\textsc{o}, \textsc{g}, \textsc{w}}
    \eqfinp
\end{align*}
In the general case, we then need to take into account pressure drop due to the flow in
the well itself through the use of a Vertical Lift Performance relationship. 

In the two cases presented in Section~\ref{sect:numerical_applications}, we can further
detail the general production function $\GeneralProductionFunction$ as follows
\begin{itemize}
\item For the gas reservoir as exposed in~\S\ref{subsect:2tanks_numerical_case}, we assume
  that the well only produces gas, and we hence obtain the following simplified
  formulation
  \begin{equation}
    \label{eq:well_gas_application_IPR}
    \gasproduced_{\timeindex} = \GeneralProductionFunction\gassuper \np{\states, \controls}
    =
    \frac{\InflowPerformanceRelationship\gassuper(
      \respressure_{\timeindex} - \presvar_{\timeindex})}{
      \gasformationfactor \np{\respressure_{\timeindex}}}
    \eqfinp
  \end{equation}
  Indeed, when we only produce gas, there is no need to consider the different
  saturations. Those saturations are necessary to find the proportion of oil, water and
  gas produced when applying a difference of pressure $\respressure - \presvar$. Having
  only gas implies that the saturations have no impact on the production.

\item When considering that the reservoir does not contain any free gas
  (i.e. $\totalreservoirgas = 0$ and $\gassaturation = 0$), we obtain the following
  simplification for the production of oil and water. We assume that the \emph{total
    production} $\totalflowvar_{\timeindex}$ follows a simplified Darcy's law
  \begin{equation}
    \totalflowvar_{\timeindex} =  \alpha \np{ \respressure_{\timeindex}
      - \presvar_{\timeindex} }
    \eqsepv \forall  \timeindex \in \timeset
    \label{eq:Darcy_s_law_appendix} \eqfinv
  \end{equation}
  where $\totalflowvar_{\timeindex}$ is given by
  \begin{equation}
    \totalflowvar_{\timeindex} =
    \oilproduced_{\timeindex} \times
    \oilformationfactor(\respressure_{\timeindex})
    + \underbrace{\gasproduced_{\timeindex} \times
      \gasformationfactor(\respressure_{\timeindex})}_{=0} +
    \waterproduced_{\timeindex} \times \waterformationfactor(\respressure_{\timeindex})
    \eqfinv
  \end{equation}
  with $\alpha$ the productivity index of the well, $\presvar_{\timeindex}$ the
  bottom-hole pressure of the well and $\totalflowvar_{\timeindex}$ the total production
  which consists of a mix of oil and water as we have assumed that we have no free gas.

  For the oil reservoir with water injection case presented in
  \S\ref{subsect:oil_wi_numerical_case}, the last assumption we make is that the amount of
  produced water is given by
  \begin{equation}
    \waterproduced_{\timeindex} \times \waterformationfactor(\respressure_{\timeindex}) =
    \alpha \np{\respressure_{\timeindex} - \presvar_{\timeindex}}
    \watercutfunct(\watersaturation_{\timeindex})
    \eqfinv
    \label{eq:FW-f-Vw}
  \end{equation}
  where $\watercutfunct$ is the water-cut function and, as already seen, where the water
  saturation is
  \[
    \watersaturation_{\timeindex} = \frac{\totalreservoirwater_{\timeindex}
    \waterformationfactor(\respressure_{\timeindex})}{\porevolume_{\timeindex}}
    \eqfinp
  \]
\end{itemize}

As we do not use more complex networks, we will not look any deeper into the network
controls and their relationship with the general production $\GeneralProductionFunction$
since those are beyond the scope of this paper.

\subsection{Reservoir dynamics}

We can now write the reservoir time evolution as a controlled dynamical system. The state
of the controlled dynamical system is
$\states = \left( \totalreservoiroil, \totalreservoirgas, \totalreservoirwater,
  \porevolume, \respressure \right)$. We also express the production volumes thanks to the
general production function, $\GeneralProductionFunction$, defined in
Equation~\eqref{eq:general_production_definition}.

% We assume that the physical model of the production network leads to a production
% function $\GeneralProductionFunction$ which relates the production volumes to the
% current state and to the network controls
% \begin{equation}
%   \left( \oilproduced, \gasproduced, \waterproduced  \right) =  \GeneralProductionFunction(\states, \controls)
%   \eqfinp
%   \label{eq:general_production_definition}
% \end{equation}
% Here, we will consider the simplified controls
% $\controls = \left( \oilproduced, \gasproduced, \waterproduced \right)$.

Now, we show that using Equations~\eqref{eq:MBalOil_bis}, \eqref{eq:MBalWater_bis},
\eqref{eq:MBalGas_bis}, \eqref{eq:PoreCompression_bis}, \eqref{eq:volume_equality_bis}
and~\eqref{eq:general_production_definition} we can build a mapping $f$ such that
$x_{t+1}=f(x_t,u_t)$ for all $\timeindex \in \timeset$. We proceed as follows: we consider
the conservation Equation~\eqref{eq:volume_equality_bis} at time $\timeindex+1$, and use
Equations~\eqref{eq:MBalOil_bis},~\eqref{eq:MBalWater_bis},~\eqref{eq:MBalGas_bis}
and~\eqref{eq:PoreCompression_bis} to obtain the equation
\begin{multline}
  \left(\totalreservoiroil_{\timeindex}-\oilproduced_{\timeindex} \right) \times
  \oilformationfactor(\respressure_{\timeindex+1}) +
  \left(\totalreservoirwater_{\timeindex} - \waterproduced_{\timeindex} \right)
  \times \waterformationfactor(\respressure_{\timeindex+1}) \\
  + \Big[ \totalreservoirgas_{\timeindex} - \gasproduced_{\timeindex} +
  \totalreservoiroil_{\timeindex} \times \left(
    \solutiongasoilratio(\respressure_{\timeindex}) -
    \solutiongasoilratio(\respressure_{\timeindex+1}) \right)
  \\
  + \oilproduced_{\timeindex} \times
  \solutiongasoilratio(\respressure_{\timeindex+1}) \Big]
  \times \gasformationfactor(\respressure_{\timeindex+1})
  \\
  = \porevolume_{\timeindex} \left(1 + \porecompressibility (\respressure_{\timeindex+1}
    - \respressure_{\timeindex}) \right)
  \eqfinv
  \label{eq:pore_equals_fluids_bis}
\end{multline}
which depends on the state and production volumes at time $\timeindex$ and of the pressure
of the reservoir at time $t+1$. As recalled in~\S\ref{sub:black-oil}, it is established
in~\cite[chap 2]{danesh} that the left-hand side of
Equation~\eqref{eq:pore_equals_fluids_bis} is a decreasing function of the reservoir
pressure $\respressure_{\timeindex+1}$. More precisely, the expansion of the oil when gas
dissolves into it due to an increase in pressure $\Delta \presvar$ is less than the
aggregated volume decrease of the free gas and the other fluids due to that same
$\Delta \presvar$. To the contrary, the right-hand side of
Equation~\eqref{eq:pore_equals_fluids_bis} is increasing with the reservoir pressure.
Hence, Equation~\eqref{eq:pore_equals_fluids_bis} gives a
function~$\Xi: \XX \times \UU \rightarrow \RR$ such that
$\forall \timeindex \in \timeset, \respressure_{\timeindex+1} = \Xi(\states_{\timeindex},
\controls_{\timeindex})$.

Moreover, note that when the PVT functions ($\oilformationfactor$, $\gasformationfactor$,
$\waterformationfactor$ and $\solutiongasoilratio$) are piecewise linear functions, the
function~$\Xi$ can be computed efficiently. We only need to look at the breaking points of
the piecewise linear functions to know on which segment we can invert
Equation~\eqref{eq:pore_equals_fluids_bis}, thus getting the reservoir pressure
$\respressure$.

Combining
Equations~\eqref{eq:MBalOil_bis},~\eqref{eq:MBalWater_bis},~\eqref{eq:MBalGas_bis},
~\eqref{eq:PoreCompression_bis} and using function~$\Xi$, we finally obtain the expression
of function~$\dynamics$ given in Equation~\eqref{eq:dynamics_expression}.

\section{Material on state reduction}
\label{sect:state_reduction}

In this section, we detail how the general dynamics can be simplified in specific cases.

\subsection{Gas reservoir state reduction}
\label{subsect:state_reduction_gas}
We consider a gas reservoir with no gas injection and where there is no water production
or extraction, as used in~\S\ref{subsect:2tanks_numerical_case}, and we prove that the
time evolution of the gas reservoir can be described by a reduced state composed of the
standard volume of gas \(\states_{\timeindex} = \totalreservoirgas_{\timeindex} \).

By assumption, the reservoir contains only gas and a constant volume of water. Thus, the
standard volume of water satisfies
$\totalreservoirwater_{\timeindex}=\totalreservoirwater_{0}$ for all
$\timeindex \in \timeset$ and the standard volume of oil satisfies
$\totalreservoiroil_{\timeindex} = 0$ for all $\timeindex \in \timeset$. Hence, the state
dimension can be reduced from dimension 5 to dimension 3.

Now, we show that the state dimension can be reduced to 1. First, we use
Equation~\eqref{eq:PoreVolume_exact} in place of the linearized
version~\eqref{eq:PoreCompression_bis} to obtain that
$\porevolume_\timeindex = \reservoirvolume_0 \mathrm{exp}\np{\porecompressibility
  \respressure_\timeindex}$ for all $\timeindex \in \timeset$. Second, we consider
Equation~\eqref{eq:volume_equality_bis} at time $\timeindex$ together with
$\totalreservoiroil_{\timeindex} = 0$ and
$\totalreservoirwater_{\timeindex}=\totalreservoirwater_{0}$ and
$\totalreservoirgas_{\timeindex} = \totalreservoirgas_{\timeindex} -
\gasproduced_{\timeindex}$ to obtain
\begin{equation}
  \totalreservoirgas_{\timeindex} \times
  \gasformationfactor(\respressure_{\timeindex}) +
  \totalreservoirwater_{0} \times
  \waterformationfactor(\respressure_{\timeindex}) =
  \reservoirvolume_0 \mathrm{exp}\np{\porecompressibility \respressure_{\timeindex}}
  \eqsepv
  \forall \timeindex \in \timeset
  \eqfinp
  \label{eq:pore_volume_gas_simplified_more}
\end{equation}
The left-hand side of Equation~\eqref{eq:pore_volume_gas_simplified_more} is a decreasing
continuous
% \jpc{pas toujours c'est des hypoth\`{e}ses a faire sur les fonctions B}
function of the pressure (the volume of gas and the production being known) which we
assume to be piecewise linear (we assume that the PVT functions are piecewise linear),
whereas the right-hand side is an increasing and continuous function of the pressure. This
implies that there can be at most one reservoir pressure which satisfies
Equation~\eqref{eq:pore_volume_gas_simplified_more}. Moreover, since the left-hand side is
piecewise linear, we can compute the reservoir pressure thanks to the $\lambertfunct$
Lambert function (the inverse relation of $f(w) = w e^w$), and since pressure is positive,
we use the $\lambertfunct_{0}$ branch of the Lambert function. Finally, we obtain a
function~$\Psi:\RR \to \RR$ such that the pressure
\begin{equation}
  \label{eq:respressure_psi}
  \respressure_{\timeindex} = \Psi(\totalreservoirgas_{\timeindex}) \eqsepv
  \forall \timeindex \in \timeset
  \eqfinv
\end{equation}
is the solution of Equation~\eqref{eq:pore_volume_gas_simplified_more}.

As the pressure, $\respressure_{\timeindex}$, is given as a function of
$\totalreservoirgas_{\timeindex}$ and the pore volume, $\porevolume_\timeindex$, is given
as a function of the pressure, $\respressure_{\timeindex}$, we obtain a reduced state of
dimension 1 given by the standard volume of gas $\totalreservoirgas_{\timeindex}$.

The only thing missing in order to get
formulation~\eqref{eq:formulation_gas_reservoir_one_tank} is to explicit the production
function. The production of gas is given by Equation~\eqref{eq:well_gas_application_IPR}.
As the reservoir pressure is given by the function $\Psi$, the production of gas when
considering a one tank reservoir is given by
\[
  \gasproduced = \frac{\InflowPerformanceRelationship\gassuper\np{\Psi(\totalreservoirgas}
    - \presvar)}{\gasformationfactor \np{\Psi(\totalreservoirgas)}}
  \eqfinp
\]
In the numerics, it is assumed that $\InflowPerformanceRelationship\gassuper$, the inflow
performance relationship of the well, is a piecewise linear function.

We consider two different models in \S\ref{subsect:2tanks_numerical_case}: a one tank
reservoir and a two tanks reservoir, as illustrated by
Figure~\ref{fig:two_coupled_reservoir_one_well}. In both cases, we have only one well and,
as the optimization is done at the bottom of the well, the unique control is given
by $\controls_{\timeindex} = \presvar_{\timeindex}$. The state in the one tank case is
$\states_{\timeindex} = \totalreservoirgas_{\timeindex}$, whereas it is
$\states_{\timeindex} = \bp{\np{\totalreservoirgas_{\timeindex}}^{(1)},
  \np{\totalreservoirgas_{\timeindex}}^{(2)}}$ for the two tanks case.

We denote by $\Psi_{1\mathbf{T}}$ the function which returns the reservoir pressure of the
one tank case given a volume of gas in the reservoir (as defined in
Equation~\eqref{eq:respressure_psi}), and $\Psi_{2\mathbf{T}}$ the function for the
producing tank pressure in the two tanks case.

The general production function $\GeneralProductionFunction_{1\mathbf{T}}$ of the one tank
case is hence given by
\begin{equation}
  \GeneralProductionFunction_{1\mathbf{T}}\gassuper(\states_{\timeindex},\controls_{\timeindex}) =
  \frac{\InflowPerformanceRelationship\gassuper \left(
      \Psi_{1\mathbf{T}}(\states_{\timeindex}) - \controls_{\timeindex} \right) }
  {\gasformationfactor(\Psi_{1\mathbf{T}}(\states_{\timeindex}))} = \gasproduced_{\timeindex}
  \eqfinp
  \label{eq:define_GPF_1T}
\end{equation}
For the two tanks case, we consider that the well only produces gas from the first tank.
The general production function $\GeneralProductionFunction_{2\mathbf{T}}$ of the two
tanks case is thus given by
\begin{equation}
  \GeneralProductionFunction_{2\mathbf{T}}\gassuper(\states_{\timeindex}, \controls_{\timeindex}) =
  \frac{\InflowPerformanceRelationship\gassuper \left(
      \Psi_ {2\mathbf{T}}^{(1)}(\states_{\timeindex}) - \controls_{\timeindex} \right) }
  {\gasformationfactor(\Psi_{2\mathbf{T}}^{(1)}(\states_{\timeindex}))}
  = \gasproduced_{\timeindex}
  \eqfinp
  \label{eq:define_GPF_2T}
\end{equation}
In the Formulation~\eqref{eq:formulation_gas_reservoir_one_tank} (for the one tank case),
we split $\GeneralProductionFunction_{1\mathbf{T}}\gassuper$ in
Constraints~\eqref{eq:Psi-mapping} and~\eqref{eq:IPR_production} to explicit the reservoir
pressure and to mirror Equation~\eqref{eq:well_gas_application_IPR}.

Moreover, since we have only one well and since the $\InflowPerformanceRelationship$
function is strictly monotonous, the production function of the well of
Equation~\eqref{eq:IPR_production} is injective. In the models considered here (one tank
or two tanks), we can thus pass from the controls to the production and from the
production to the controls without any ambiguity at a given state: the function
$\GeneralProductionFunction\gassuper(\states, \cdot)$ is a bijection, hence we find the
(unique) bottom-hole pressure associated with a given production $\gasproduced$ when in
state $\states$. Finally, we obtain the admissibility set of the gas reservoir case. As
the gas production $\gasproduced_{\timeindex}$ must be nonnegative, we obtain that the
control must satisfy
$ \presvar_{\timeindex} \in \left[ 0, \respressure_{\timeindex} \right]$ for all time
$\timeindex \in \timeset$, which gives the admissible control set
\begin{equation}
  \admcontrolset(\states_{\timeindex}) = \left[ 0, \respressure_{\timeindex}
  \right] = \left[ 0, \Psi_{1\mathbf{T}}(\states_{\timeindex}) \right] \eqfinp
  \label{eq:define_adm_control_set_1T}
\end{equation}

\subsection{Oil reservoir with water injection state reduction}

Now, we consider an oil reservoir where water injection is used to keep the reservoir
pressure constant as in \S\ref{subsect:oil_wi_numerical_case}. To eliminate the
possibility of having free gas in the reservoir, we assume that the initial pressure in
the reservoir is above the bubble-point. Indeed, as we are going to keep the pressure
constant, the pressure will always remain above the bubble-point.

We assume that the \emph{produced water}\footnote{Here, the produced water
  $\waterproduced$ is the water that is produced from the well. It should not be confused
  with the net produced water, which is the difference $\waterproduced - \waterinjected$
  between the water produced and the water injected} is given by
Equation~\eqref{eq:FW-f-Vw}.

We now prove that the standard volume of water $\totalreservoirwater_{\timeindex}$ may be
used as a state for describing the reservoir dynamics. To start with, we have that
$\totalreservoirgas_{\timeindex} = 0$, $\gasproduced_{\timeindex}= 0$ and
$\respressure_{\timeindex}= \respressure_{0}$ for all $\timeindex \in \timeset$. Moreover,
using Equation~\eqref{eq:PoreVolume_exact} in place of the linearized
version~\eqref{eq:PoreCompression_bis} we obtain that the pore volume is constant over
time and given by
\( \porevolume_{\timeindex} = \reservoirvolume_0 \mathrm{exp}\np{\porecompressibility
  \respressure_{0}} \). Hence, the state dimension can be reduced from dimension~5 to
dimension~2 as \(\totalreservoirgas_{\timeindex}\), \(\respressure_{\timeindex}\) and
\(\porevolume_\timeindex\) are known over time.

Now, using Equation~\eqref{eq:volume_equality_bis} combined with the fact that
$\totalreservoirgas_{\timeindex} = 0$, we obtain that
\begin{equation}
  \totalreservoiroil_{\timeindex} \times
  \oilformationfactor(\respressure_{0}) +
  \totalreservoirwater_{\timeindex} \times
  \waterformationfactor(\respressure_{0}) =
  \porevolume_{0}
  \eqfinv
  \forall \timeindex \in \timeset
  \eqfinp
  \label{eq:volume_equality_for water_injection}
\end{equation}
Thus, the standard volume of oil in
the reservoir is obtained as a function of the standard volume of water as follows
\[
  \totalreservoiroil_{\timeindex} = \frac{\porevolume_{0} - \totalreservoirwater_{\timeindex}
    \times \waterformationfactor(\respressure_0)}{\oilformationfactor(\respressure_0)}
  \eqfinp
\]
Moreover, using Equation~\eqref{eq:Darcy_s_law_appendix} and Equation~\eqref{eq:FW-f-Vw},
for all time $\timeindex \in \timeset$, we have that
\begin{subequations}
  \begin{align}
  \waterproduced_{\timeindex}
  &= \GeneralProductionFunction\watersuper \np{\totalreservoirwater_{\timeindex},
    \presvar_{\timeindex}}
  \\
  \text{with}&\; \GeneralProductionFunction\watersuper \np{\totalreservoirwater, \presvar}
    =
    \frac{\alpha \np{\respressure_{0} - \presvar}
               \watercutfunct\Bp{\frac{\totalreservoirwater}{\porevolume_{0}}
               \waterformationfactor(\respressure_{0})}}
    {\waterformationfactor(\respressure_{0})}
    \eqfinv
    \label{eq:Phiw}
    % \\
  \end{align} 
\end{subequations} 
  and
\begin{subequations}
  \begin{align}
  \oilproduced_{\timeindex}
  &= \GeneralProductionFunction\oilsuper \np{\totalreservoirwater_{\timeindex},
    \presvar_{\timeindex}}
  \\
  \text{with}&\;
    \GeneralProductionFunction\oilsuper \np{\totalreservoirwater, \presvar}
    =
    \frac{\alpha \np{\respressure_{0} - \presvar}
    \bgp{1- \watercutfunct\Bp{\frac{\totalreservoirwater}{\porevolume_{0}}
    \waterformationfactor(\respressure_{0})}}}
    {\oilformationfactor(\respressure_{0})}
    \label{eq:Phio}
  \eqfinp
  \end{align}
\label{eq:def_Fo_WI}
\end{subequations}

Now, we turn to the time evolution of the standard volume of water.
Equation~\eqref{eq:MBalWater_bis} must be changed as we need to introduce the injected
water $\waterinjected_{\timeindex}$ at time $\timeindex$ to obtain
\begin{equation}
  \totalreservoirwater_{\timeindex+1} = \totalreservoirwater_{\timeindex} -
  \waterproduced_{\timeindex} + \waterinjected_{\timeindex}
  \eqsepv \forall \timeindex \in \timeset
  \eqfinp
  \label{eq:MBalWater_bis_injection}
\end{equation}
It remains to show that the water injection can be deduced from the previous equations.
Using Equation~\eqref{eq:volume_equality_bis} at time ${\timeindex+1}$ combined with
Equation~\eqref{eq:MBalWater_bis_injection} and Equation~\eqref{eq:MBalOil_bis} gives
\begin{equation}
  \np{\totalreservoirwater_{\timeindex} - {\waterproduced_{\timeindex}} +
    \waterinjected_{\timeindex}} \times \waterformationfactor \left(
    \respressure_{0} \right) + \\
  \bp{\totalreservoiroil_{\timeindex} -\oilproduced_{\timeindex}} \times
  \oilformationfactor \np{\respressure_{0}} = \porevolume_{0}
  \eqfinv
\end{equation}
which, using Equation~\eqref{eq:volume_equality_for water_injection}, \eqref{eq:Phiw}
and~\eqref{eq:Phio}, gives
\begin{equation}
  \waterinjected_{\timeindex}
  = \waterproduced_{\timeindex}
  + \oilproduced_{\timeindex} \times \frac{ \oilformationfactor \np{\respressure_{0}}}{
    \waterformationfactor \left( \respressure_{0} \right)}
  = \frac{\alpha \np{\respressure_{0} - \presvar_{\timeindex}}}
    {\waterformationfactor(\respressure_{0})}
    \eqfinp
    \label{eq:def_Fwi_function}
\end{equation}
We conclude that we obtain a state dynamics with a one dimensional state
$x_t =\totalreservoirwater_{\timeindex}$, a one dimensional control
$u_t=\presvar_{\timeindex}$, and state dynamics given by
\begin{equation}
  \totalreservoirwater_{\timeindex+1} = \totalreservoirwater_{\timeindex} -
  \frac{\alpha \np{\respressure_{0} - \presvar_{\timeindex}}
    \Bp{\watercutfunct\bp{\totalreservoirwater_{\timeindex}
        \waterformationfactor(\respressure_{0})/{\porevolume_{0}}}
    -1}}
    {\waterformationfactor(\respressure_{0})}
    %\eqsepv \forall \timeindex \in \timeset
    \eqfinp
\end{equation}

% \section{Details on the impact of the states and controls discretizations}
% \label{sect:discretization_impact}
% The Appendix~\ref{sect:discretization_impact} is available free of charge via the Internet.

% \section{Additional material on the decline curves formulation}
% \label{sect:decline_curve_explanation}
% The Appendix~\ref{sect:decline_curve_explanation} is available free of charge via the Internet.

\section{Details on the impact of the states and controls discretizations}
\label{sect:discretization_impact}

\paragraph{One tank gas reservoir.}
In the application of \S\ref{subsubsect:one_tank_gas}, we tried different discretization
values for the state and control spaces. Results get better each time we increase the
number of states or controls used in the loops of Algorithm~\ref{alg:dynamic_programming}.
The optimal values and CPU times are compiled in
Table~\ref{tab:discretization_table_one_tank}. Discretization of the control space has
less impact than discretization of the state space (there is no significant improvement
when using more than 10 possible controls). We used 50 possible controls for the rest of
the state discretization analysis to ensure we do not have any issues due to the control
space. Moreover, the computation time grows linearly with the number of controls, hence we
only got penalized by a factor of 5 for the computation time compared to being at the most
efficient level for the discretization of the controls. We can also remark that going
beyond \numprint{10000} points for the state discretization yields no discernible
improvement (less than $0.2\%$). However, the computation time grows exponentially with
the state discretization. We hence used \numprint{10000} points for the states and 20
controls for the results presented in \S\ref{subsubsect:one_tank_gas}.

\begin{table}[htbp]
  \begin{center}
    \begin{tabular}{rrr}
      \toprule
      State discretization & Value (M\euro) & CPU time (s)           \\
      \midrule
      \numprint{100}       &       602      &     \numprint{1.25}    \\
      \numprint{200}       &       689      &     \numprint{1.45}    \\
      \numprint{500}       &       725      &     \numprint{2.50}    \\
      \numprint{1000}      &       736      &     \numprint{7.50}    \\
      \numprint{2000}      &       740      &     \numprint{25.20}   \\
      \numprint{5000}      &       742      &     \numprint{110.00}  \\
      \numprint{10000}     &       743      &     \numprint{653.00}  \\
      \numprint{20000}     &       743      &     \numprint{2288.00} \\
      \numprint{50000}     &       743      &     \numprint{8142.00} \\
      \bottomrule
    \end{tabular}
    \caption{Summary of the impact of the discretization of the state space on the one
      tank formulation, with 50 possible controls\label{tab:discretization_table_one_tank}}
  \end{center}
\end{table}

\paragraph{Two tanks gas reservoir.}
We tried different discretization values for the two reservoirs problem of
\S\ref{subsubsect:two_tanks_gas}: $200 \times 200$ (i.e. the two reservoirs are discretized
with 200 points each), $400 \times 400$, $600 \times 600$ and
$\numprint{1000} \times \numprint{1000}$. Results are summarized in
Table~\ref{tab:discretization_2tanks_DP}, which shows the computation time of the
optimization and the optimal value obtained. As can be seen, the computation time grows
exponentially with the discretization, as we need to handle more and more values when we
get a finer discretization. However, performance remains reasonable for the number of time
steps considered. We can also remark that going past a $200 \times 200$ discretization of
the states of the reservoir does not improve the optimal value. A very small impact is
observed from the discretization of the controls. Indeed, almost no improvement is
obtained above 10 possible controls (we hence used $50$ possible controls in
Table~\ref{tab:discretization_2tanks_DP} to ensure the discretization of the controls will
not influence the analysis of the discretization of the states). All the results of
\S\ref{subsubsect:two_tanks_gas} were therefore computed with the $400 \times 400$
discretization for the states, and $20$ for the controls.

\begin{table}[htbp!]
  \begin{center}
    \begin{tabular}[bct]{crr}
      \toprule
      State discretization    & CPU time (s) & Value (M\euro) \\
      \midrule
      $50 \times 50$          &   $\numprint{5.1}$      &  $730$     \\
      $100 \times 100$        &   $\numprint{28.3}$     &  $735$     \\
      $200 \times 200$        &   $\numprint{115.3}$    &  $736$     \\
      $400 \times 400$        &   $\numprint{706.0}$      &  $736$     \\
      $600 \times 600$        &   $\numprint{3893.0}$     &  $736$     \\
      $1000 \times 1000$      &   $\numprint{18089.0}$    &  $736$     \\
      \bottomrule
    \end{tabular}
    \caption{Impact of the discretization of the state space on the two tanks model, with
      50 possible controls\label{tab:discretization_2tanks_DP}}
    \end{center}
\end{table}

\paragraph{Oil reservoir with water injection.}
We tried different values for the discretization of the state space of the problem
described in \S\ref{subsect:oil_wi_numerical_case}. However, the discretization of the
controls had no impact, as the controls only took two different values: either no
production, or production at the maximal rate. We therefore chose 10 possible controls to
ensure we do not missed another behavior during the analysis on the impact of the
discretization of the states. Table~\ref{tab:results_oil_with_wi} compiles the time to
solve and the associated results of the optimization depending on the number of points
considered for the discretization of the states space. We note that there is not a lot of
gain from going from \numprint{10000} points to \numprint{100000} points in the
discretization, whereas computation time grows by more than $100$ times.

\begin{table}[htbp!]
  \begin{center}
    \begin{tabular}[bct]{cccc}
      \toprule
      Discretization & Time steps & CPU time (s) & Value (M\euro) \\
      \midrule
      \numprint{1000}        &   $240$    &   $0.35$       &  $3182$    \\
      \numprint{10000}        &   $240$    &   $12.05$      &  $3358$    \\
      \numprint{100000}       &   $240$    &   $1575$       &  $3376$    \\
      \bottomrule
    \end{tabular}
    \caption{Summary of the dynamic programming results for the oil reservoir with
      water injection\label{tab:results_oil_with_wi}}
  \end{center}
\end{table}

\section{Additional material on the decline curves formulation}
\label{sect:decline_curve_explanation}
Usually, formulations using decline curves, as can be seen in the works
of~\cite{Grossmann98}, are of the form:
\begin{subequations}
  \begin{align}
    \max_{
    \controls} ~ ~
    &
      \sum_{\timeindex=0}^{\horizon}
      \rho^{\timeindex} \costfunct_{\timeindex}(\controls_{\timeindex})
      \label{eq:obj_dc_bis}\\
    s.t. ~~
    &
      \oilproduced_{\timeindex} \leq \declinecurve \left( \sum_{s = 0}^{\timeindex-1}
      \oilproduced_{\cumulatedtimeindex} \right)
      \eqsepv \forall \timeindex \in \timeset \setminus \{0\}
      \eqfinv
      \label{eq:production_dc_bis}\\
    & \controls_{\timeindex} \in \admcontrolset_{\timeindex}\left(
      \sum_{s = 0}^{\timeindex-1} \oilproduced_{\cumulatedtimeindex}
      \right)
      \eqsepv \forall \timeindex \in \timeset
      \eqfinp
      \label{eq:controls_admissibility_dc_bis}
  \end{align}
  % % & \sigma(\bar{\controlsva}_{\timeindex}) \subset
  % % \sigma(\observerva_{0}, ..., \observerva_{\timeindex})
  % % \eqsepv \forall \timeindex \in \timeset
  \label{eq:general_formulation_decline_curves_bis}
\end{subequations}
Using decline curves, or oil deliverability curves, means using
Equation~\eqref{eq:production_dc_bis} to predict the reservoir's behavior. It states that
the maximal rate at time $\timeindex$ only depends on the oil cumulated production until
time~$\timeindex$. In the general case, there is no reason to believe that there is an
equivalence between a material balance model for the reservoir and a decline curve
represented with a function~$\declinecurve$.

However, when the state of the material balance formulation can be reduced to a one
dimensional state (such as a reservoir which only contains gas), there can be an
equivalence between the decline curve and the material balance formulations, as was stated
in Proposition~\ref{prop:equivalance_dc_mbal_1state}.

\begin{proof}[Proof of Proposition~\ref{prop:equivalance_dc_mbal_1state}]
  Let us consider the component $\GeneralProductionFunction\gassuper: \XX \times \UU
  \rightarrow \RR$ of the production mapping $\GeneralProductionFunction:\XX \times \UU
  \rightarrow \RR^3$ such that
  \begin{equation}
    \gasproduced_{\timeindex} = \GeneralProductionFunction\gassuper \left(
      \states_{\timeindex}, \controls_{\timeindex} \right)
    \eqfinp
    \label{eq:production_function_def}
  \end{equation}
  Therefore, we have
  \begin{equation}
    \gasproduced_{\timeindex} \leq \max_{\controls}\GeneralProductionFunction\gassuper
    \left( \states_{\timeindex}, \controls \right)
    \eqfinp
    \label{eq:max_production_def}
  \end{equation}
  Moreover, having a one-dimensional state greatly simplifies the dynamics, as we only
  need to consider one fluid. The dynamics thus simplifies to
  \begin{equation}
    \states_{\timeindex+1} = \dynamics\left( \states_{\timeindex}, \controls_{\timeindex}
    \right) = \states_{\timeindex} - \gasproduced_{\timeindex}
    \eqfinp
    \label{eq:one_dimension_dynamics}
  \end{equation}
  By propagating the simplified dynamics~\eqref{eq:one_dimension_dynamics} and by
  re-injecting it in Equation~\eqref{eq:max_production_def}, we get:
  \begin{equation}
    \gasproduced_{\timeindex} \leq \underbrace{\max_{\controls}
      \GeneralProductionFunction\gassuper
      \left(\states_{0} - \sum_{\cumulatedtimeindex=0}^{\timeindex-1}
        \gasproduced_{\cumulatedtimeindex}, \controls \right)}_{\declinecurve(
      \sum_{\cumulatedtimeindex=0}^{\timeindex-1}\gasproduced_{\cumulatedtimeindex})}
    \eqfinp
    \label{eq:decline_curve_1D}
  \end{equation}
  Hence, Equation~\eqref{eq:decline_curve_1D} defines the function~$\declinecurve$. The
  equivalence exists when the state is reduced to one dimension (as similar reasoning can
  be applied to the other one-dimensional cases).
\end{proof}

However, when considering more complex cases, such as a reservoir with both oil and gas, or
when there is water encroachment (influx of water in the reservoir from the aquifer), we
cannot have a reduction to a one-dimensional state. Decline curves, or oil deliverability
curves, will not be equivalent to the material balance system, as they can only represent
a one dimensional dynamical system, where the state is the cumulated production.

Even if we have a state that cannot be reduced to one dimension, we can still propagate the
dynamics in Equation~\eqref{eq:production_function_def}:
\begin{align*}
  \gasproduced_{\timeindex}
  &= \GeneralProductionFunction\gassuper(\states_{\timeindex}, \controls_{\timeindex}) \\
  &
  = \GeneralProductionFunction\gassuper \left(
      \dynamics\left( \dynamics \left( \dots \dynamics \left(\states_{0}, \controls_{0}
              \right), \dots \right), \controls_{\timeindex-1} \right),
      \controls_{\timeindex} \right)
  \eqfinp
\end{align*}
However, there is no reason to believe that there exists a function $\declinecurve$
depending on the sum of productions in the general case, contrarily to the one-dimensional
case. This is why those functions are generated with a given production planning, i.e. a
series of controls applied to the reservoir. Given a series of admissible controls
$\controlseries = \np{ \tilde{\controls}_0, \dots , \tilde{\controls}_{\horizon}}$, one
can create an oil-deliverability curve, that takes as argument the total cumulated
production and returns the maximal possible production. It however depends on the
underlying production planning $\controlseries$. We can create such function
$\tilde{\declinecurve}_{\controlseries}$ through the
Algorithm~\ref{alg:oil_deliverability_curve}.
\begin{algorithm}[htp!]
  \SetAlgoLined
  control\_to\_apply = $\controlseries$\;
  current\_state = $\states_0$\;
  cumulated\_production = 0\;
  max\_production = $\max_{\controls} \GeneralProductionFunction\gassuper$(current\_state,
  $\controls$)\;
  list\_of\_points = \{(cumulated\_production, max\_production)\}\;
  \For{t from 1 to $\horizon$}
  {
    $\tilde{\controls}$ = control\_to\_apply[t]\;
    production = $\GeneralProductionFunction\gassuper$(current\_state,
    $\tilde{\controls}$)\;
    cumulated\_production = cumulated\_production + production\;
    current\_state = \dynamics(current\_state, $\tilde{\controls}$)\;
    max\_production = $\max_{\controls} \GeneralProductionFunction\gassuper$(
    current\_state, $\controls$)\;
    push(list\_of\_points, (cumulated\_production, max\_production))\;
  }
  \Return{list\_of\_points}
  \caption{Finding the points of the piecewise linear function
    $\tilde{\declinecurve}_{\controlseries}$}
  \label{alg:oil_deliverability_curve}
\end{algorithm}

Once we have a list of points of $\tilde{\declinecurve}_{\controlseries}$, we consider
a linear interpolation between those points as the decline curve we use in the
optimization problem~\eqref{eq:general_formulation_decline_curves}.

In~\citep{Grossmann98,Marmier}, the authors use decline curves, i.e.
oil-deliverability curves with natural depletion at the maximal rate. This means that
there is no injection, and the production planning consists of maximal production rates.
We can generate those decline curves with a tweaked version of the previous procedure (see
Algorithm~\ref{alg:decline_curve}).
\begin{algorithm}[htp!]
  \SetAlgoLined
  current\_state = $\states_0$\;
  cumulated\_production = 0\;
  max\_production = $\max_{\controls} \GeneralProductionFunction\gassuper$(current\_state,
  $\controls$)\;
  list\_of\_points = \{(cumulated\_production, max\_production)\}\;
  \For{t from 1 to $\horizon$}
  {
    $\tilde{\controls} = \argmax_{\controls} \GeneralProductionFunction\gassuper$(
    current\_state, $\controls$)\;
    production = $\GeneralProductionFunction\gassuper$(current\_state,
    $\tilde{\controls}$)\;
    cumulated\_production = cumulated\_production + production\;
    current\_state = \dynamics(current\_state, $\tilde{\controls}$)\;
    max\_production = $\max_{\controls} \GeneralProductionFunction\gassuper$(current\_state,
    $\controls$)\;
    push(list\_of\_points, (cumulated\_production, max\_production))\;
  }
  \Return{list\_of\_points}
  \caption{Finding the points of the piecewise linear function $\declinecurve$}
  \label{alg:decline_curve}
\end{algorithm}

\end{document}